\documentclass [12pt,oneside]{amsart}
\usepackage{xcolor}
\usepackage{amssymb,amsmath,amsthm,amsfonts}
\usepackage{enumerate}
\usepackage{setspace}
\usepackage{empheq}
\usepackage[all]{xy}
\usepackage{verbatim}
\usepackage[normalem]{ulem}
\usepackage{todonotes}
\usepackage[bookmarks,colorlinks,breaklinks]{hyperref}  
\hypersetup{linkcolor=blue,citecolor=blue,filecolor=blue,urlcolor=blue} 

\numberwithin{equation}{section}
\numberwithin{figure}{section}

\newtheorem{theorem}{Theorem}[section]
\newtheorem{proposition}[theorem]{Proposition}

\newtheorem{lemma}[theorem]{Lemma}

\theoremstyle{definition}

\newtheorem{example}[theorem]{Example}
\newtheorem{remark}[theorem]{Remark}
\newtheorem{question}[theorem]{Question}

\definecolor{myblue}{rgb}{0.6, 0.9, 1}

\makeatletter

\newcommand{\Rmnum}[1]{\expandafter\@slowromancap\romannumeral #1@}
\makeatother

\definecolor{myblue}{rgb}{0.6, 0.9, 1}
\definecolor{mygreen}{rgb}{0,0,1}
\definecolor{purple}{rgb}{0.6,0.2,1}
\definecolor{orange}{rgb}{0.8,0,0.2}

\newcommand{\bC}{\mathbb{C}}
\newcommand{\bP}{\mathbb{P}}
\newcommand{\C}{\mathbb{C}}
\newcommand{\Q}{\mathbb{Q}}

\newcommand{\bQ}{\mathbb{Q}}

\newcommand{\bD}{\mathbb{D}}

\newcommand{\eps}{\varepsilon}

\newcommand{\la}{\lambda}

\newcommand{\al}{\alpha}

\newcommand{\Qbar}{\overline{\bQ}}

\newcommand{\supp}{\operatorname{supp}}
\newcommand{\cL}{\mathcal{L}}

\newcommand{\<}{\langle}
\renewcommand{\>}{\rangle}

\newcommand\Mbar {\overline{M}}

\newcommand\iso{\simeq}

\newcommand{\Z}{\mathbb{Z}}

\newcommand\M{\mathrm{M}}
\newcommand\PCF{\mathrm{PCF}_d}
\newcommand{\cX}{\mathcal{X}}
\newcommand{\cC}{\mathcal{C}}
\newcommand{\cY}{\mathcal{Y}}
\newcommand{\cM}{\mathcal{M}}
\newcommand{\Tbif}{T_{\mathrm{bif}}}

\newcommand{\Per}{\mathrm{Per}}
\newcommand{\Ch}{\mathrm{Ch}}
\newcommand{\hcrit}{\hat{h}_{\mathrm{crit}}}
\newcommand{\Spec}{\operatorname{Spec}}

\begin{document}

\title[Bounded geometry]{Bounded geometry for PCF-special subvarieties}
	
\author{Laura De Marco, Niki Myrto Mavraki, and Hexi Ye}
\email{demarco@math.harvard.edu}
\email{myrto.mavraki@utoronto.ca}
\email{yehexi@zju.edu.cn}

	
\begin{abstract}
For each integer $d\geq 2$, let $\M_d$ denote the moduli space of maps $f: \bP^1\to \bP^1$ of degree $d$.  We study the geometric configurations of subsets of postcritically finite (or PCF) maps in $\M_d$.  A complex-algebraic subvariety $Y \subset \M_d$ is said to be PCF-special if it contains a Zariski-dense set of PCF maps.  Here we prove that there are only finitely many positive-dimensional irreducible PCF-special subvarieties in $\M_d$ with degree $\leq D$.  In addition, there exist constants $N = N(D,d)$ and $B = B(D,d)$ so that for any complex algebraic subvariety $X \subset \M_d$ of degree $\leq D$, the Zariski closure $\overline{X \cap \PCF}$ has at most $N$ irreducible components, each with degree $\leq B$.  We also prove generalizations of these results for points with small critical height in $\M_d(\Qbar)$.
\end{abstract}	
	
	
	\maketitle

\section{Introduction}

For each integer $d\geq 2$, let $\M_d$ denote the moduli space of maps $f: \bP^1\to \bP^1$ of degree $d$, defined over the field $\C$ of complex numbers.  The space $\M_d$ consists of M\"obius conjugacy classes of maps $f$; it is constructed as a complex orbifold in \cite{Milnor:quad} and as a scheme over $\Spec \Z$ in \cite{Silverman:Ratd}. Let $\PCF\subset \M_d$ denote the subset of postcritically finite maps; i.e., the conjugacy classes of maps $f$ for which each critical point has a finite forward orbit.  We say that an irreducible complex algebraic subvariety $X \subset \M_d$ is {\bf PCF-special} if $\PCF\cap X$ is Zariski dense in $X$. Note that $\M_d$ is itself PCF-special, as is the polynomial subvariety $\mathrm{MPoly}_d \subset \M_d$, as are all subvarieties defined by orbit relations among the critical points; see for example \cite[Proposition 2.6]{BD:polyPCF} and \cite[Theorem 6.2]{D:stableheight}.  Conjecturally, all PCF-special subvarieties are determined by critical orbit relations (in a general sense, accounting for symmetries among the maps) \cite[Conjecture 1.10]{BD:polyPCF}.  

In this article, we study the geometry of the set $\PCF$ by examining how it intersects with families of varieties in $\M_d$, and we prove uniform bounds on the degrees of PCF-special subvarieties in these families.  Here, and throughout this article, we measure degree with respect to a fixed choice of very ample line bundle on $\M_d$, for example providing the projective compactification of $\M_d$ defined in \cite{Silverman:Ratd}.  The degree bounds will depend on the choice, but the following result holds for any such choice:

\begin{theorem} \label{main1}
Fix integers $d\geq 2$ and $D\geq 1$.  There exist constants $B = B(D,d)$ and $N = N(D,d)$ so that for any complex algebraic subvariety $X \subset \M_d$ of degree $\leq D$, the Zariski closure 
	$$\overline{X \cap \PCF}$$ 
has at most $N$ irreducible components, each with degree $\leq B$.  
\end{theorem}

\begin{theorem} \label{main2}
Fix integers $d\geq 2$ and $D \geq 1$.  There are only finitely many positive-dimensional irreducible PCF-special subvarieties in $\M_d$ with degree $\leq D$.    
\end{theorem}

\begin{example} \label{Pern}
Consider the sequence of PCF-special curves $\Per_n(0)$ in $\M_2$, defined by the condition that one critical point of a quadratic rational map has exact period $n$. These were introduced by Milnor in \cite{Milnor:quad} and conjectured to be irreducible for every $n$.  While irreducibility remains open, Theorem \ref{main2} implies that the degrees of the irreducible components of $\Per_n(0)$ must be growing with $n$.  Theorem \ref{main2}  implies degree growth also of the irreducible components of the analogous hypersurfaces $\Per_n(0)$ in $\M_d$, for every $d\geq 2$.
\end{example}

\begin{example} \label{Per1}
Consider the family of curves $\Per_1(\lambda)$ in $\M_2$, parameterized by $\lambda \in \C$, consisting of all quadratic maps having a fixed point of multiplier $\lambda$.  This family was first introduced by Milnor in \cite{Milnor:quad}, and they are lines in the natural coordinates that Milnor introduced on $\M_2 \iso \C^2$.  It is known that $\Per_1(\lambda)$ is PCF-special if and only if $\lambda = 0$ \cite{DWY:QPer1}.  And although $\PCF$ is Zariski-dense in $\M_2$, Theorem \ref{main1} combined with the main result of \cite{DWY:QPer1} implies: {\em There is a uniform bound $N$ so that 
	$$\# \, \Per_1(\lambda) \cap \PCF \; \leq \; N$$
for all $\lambda \not= 0$ in $\C$.}  Theorem \ref{main1} also implies the corresponding result for the family of curves $\mathrm{Per}_n(\lambda)$ in the moduli space $\mathrm{MPoly}_3$ of cubic polynomials,  for any fixed $n\geq 1$, in view of \cite[Theorem 1.1]{BD:polyPCF} \cite[Theorem B]{Favre:Gauthier:cubics}.
\end{example}

In fact we prove stronger versions of Theorem \ref{main1} and Theorem \ref{main2} in the spirit of the Bogomolov Conjecture \cite{Ullmo:Bogomolov, Zhang:Bogomolov} or, more recently, of the ``height gap principle'' \cite{Gao:Ge:Kuhne, Gao:survey} in the setting of abelian varieties, for subvarieties of $\M_d$ defined over $\Qbar$.  The {\bf critical height} of a conjugacy class $[f]\in \M_d(\Qbar)$ is defined by
	$$\hcrit(f)=\sum_{\{c~:~f'(c)=0\}}\hat{h}_f(c) \; \ge \; 0,$$
for any representative $f$ defined over $\Qbar$, where the critical points are counted with multiplicity and $\hat{h}_f$ is the Call-Silverman canonical height on $\bP^1(\Qbar)$ \cite{Call:Silverman} \cite{Silverman:moduli}.  Ingram proved that $\hcrit$ is comparable to a Weil height on the open subset $\M_d \setminus \cL_d$ where $\cL_d$ is the locus of flexible Latt\`es maps \cite{Ingram:criticalheight}. A more explicit version of Ingram's result was later established in \cite[Theorem C]{Gauthier:Okuyama:Vigny}.  A different approach is given by the Yuan--Zhang theory of adelic line bundles; see \cite[Remark 6.3.3]{Yuan:Zhang:quasiprojective}.  
The PCF points in $\M_d(\Qbar)$ are precisely those with $\hcrit(f) = 0$. 

To state our results, we fix a projective embedding of $\M_d$ defined over $\Q$.  We also fix height functions $h_{\Ch}$ on the associated Chow varieties $\Ch(\M_d, r, D)$ of cycles of dimension $r$ and degree $\leq D$ in $\M_d$, for example a Philippon height \cite{Philippon1}; see \S\ref{chow} for definitions.  The constants $B$, $N$, and $\epsilon$ of the following theorems will depend on these choices.

\begin{theorem} \label{main1Bogomolov}
Fix integers $d\geq 2$ and $D\geq 1$.  There exist constants $B = B(D,d)$, $N= N(D,d)$, and $\epsilon=\epsilon(D,d) >0$ so that for any irreducible algebraic subvariety $X \subset \M_d$ of degree $\leq D$ defined over $\Qbar$, and every $0\le\epsilon'<\epsilon$, the Zariski closure 
 	$$\overline{X \cap \{z\in \M_d(\Qbar)~:~\hcrit(z) \leq \epsilon'\max\{1,h_{\Ch}(X)\}\}}$$ 
has at most $N$ irreducible components, each with degree $\leq B$.  
\end{theorem}

The {\bf essential minimum} of the critical height on a subvariety $Y\subset\M_d$ defined over $\Qbar$ is 
$$\mathrm{ess.min.}(\hcrit|_Y) \; := \;\sup_{Y'\subset Y} \inf_{y \in (Y\setminus Y')(\Qbar)} \hcrit(y)  \; \ge \; 0,$$
where the supremum is taken over all proper subvarieties $Y'$ in $Y$ defined over $\Qbar$.  Note that all PCF-special subvarieties $Y\subset \M_d$ are defined over $\Qbar$ and have essential minimum 0, because they contain a Zariski-dense set of PCF points that can be defined over $\Qbar$; see \S\ref{Thurston}. 
We also prove the following result, strengthening Theorem \ref{main2}. 

\begin{theorem} \label{main2Bogomolov} 
Fix integers $d\geq 2$ and $D \geq 1$. There exists $\epsilon''=\epsilon''(d,D) > 0$ such that there are only finitely many positive-dimensional irreducible subvarieties $Y\subset \M_d$ defined over $\Qbar$ with degree $\leq D$ and 
$\mathrm{ess.min.}(\hcrit|_Y) \leq \epsilon''\max\{1,h_{\Ch}(Y)\}.$
 \end{theorem}

\subsection{History and context}
The study of the PCF-special varieties in $\M_d$ grew out of parallels between the theory of complex dynamical systems and the study of abelian varieties.  A conjectural classification of PCF-special subvarieties in $\M_d$ first appears in \cite[Chapter 6]{Silverman:moduli} where it was given the name of ``Dynamical Andr\'e-Oort" (or ``DAO" as it is known in \cite{Ji:Xie:DAO}), because the geometry of the PCF locus in $\M_d$ shares features with the locus of CM points in a moduli space of abelian varieties; see \cite{Zannier:book} for background on several Unlikely Intersection problems.  The first cases of DAO were proved by Baker-DeMarco \cite{BD:polyPCF} and Ghioca-Hsia-Tucker \cite{Ghioca:Hsia:Tucker}, followed by a series of results mostly focused on polynomial dynamics where the relations and symmetries are understood, such as \cite{Ghioca:Ye:cubics, Favre:Gauthier:cubics}, culminating in a complete classification of PCF-special curves in moduli spaces of polynomials by Favre-Gauthier \cite{Favre:Gauthier:book} and a characterization of PCF-special curves in $\M_d$ by Ji-Xie \cite{Ji:Xie:DAO}.  We note, however, that there are currently no known classification results for PCF-special subvarieties of dimensions $> 1$, even when restricting to spaces of polynomials.  Nevertheless, Theorem \ref{main1} can be viewed as a ``Uniform DAO" that holds for PCF-special subvarieties of any dimension.

The questions addressed here are inspired by Pink's conjectures \cite{Pink:conjecture} in the setting of families of abelian varieties, and the uniformity results about subgroup schemes in families of abelian varieties, for example as obtained recently in \cite{Kuhne:UML, DGH:uniformity, Gao:Ge:Kuhne, Gao:Habegger:RMM, Yuan:uniform}.  The analogous study of preperiodic subvarieties for regular self-maps of projective spaces (or more generally, for polarized dynamical systems, which then includes the geometry of subgroups in abelian varieties) was initiated by Zhang (see, for example, \cite{Zhang:distributions}).  Uniform bounds on the geometry of preperiodic points appeared in \cite{DKY:UMM, DKY:quad, DM:Mandelbrot, Mavraki:Schmidt}.

\begin{remark}
We emphasize that the aforementioned uniformity results each relied crucially on the classification of the `special' varieties in their respective settings. In contrast, the classification of PCF-special subvarieties is only conjectural \cite{BD:polyPCF}; we obtain uniformity without proving that PCF-special subvarieties are necessarily defined by critical orbit relations.  In particular, we do not recover (or use) the main result of \cite{Ji:Xie:DAO} for curves in $\M_d$, and each of our theorems is new also for curves in $\M_d$.  We do not know of any other result in the literature where uniformity is known without the classication.
\end{remark}

A far-reaching conjecture in \cite{DM:UDMM} implies the classification of PCF-special subvarieties (as formulated in \cite[Conjecture 1.10]{BD:polyPCF}) and the uniform bounds on their geometry (our Theorems \ref{main1} and \ref{main2}) and many results about geometry of torsion points in abelian varieties as special cases. It also contains the (partial) uniformity results obtained in \cite{DM:commonprep} and \cite{Gauthier:Taflin:Vigny}.  

Finally, we remark that we do not give explicit bounds on $B$ and $N$ in Theorem \ref{main1}, nor on the number of PCF-special subvarieties in Theorem \ref{main2}.  It would be very interesting to know how large these values must be and how they depend on $d$ and $D$ (and the dimension of the PCF-special subvarieties). In different contexts, some effective bounds have been given in \cite{DKY:quad, David:Philippon:2007, DM:Mandelbrot}.

\subsection{Proof strategy} \label{strategy}
We build upon the theory of arithmetic equidistribution for heights on quasiprojective varieties \cite{Kuhne:UML, Gauthier:goodheights, Yuan:Zhang:quasiprojective}, and we make use of complex-dynamical tools for analyzing bifurcation measures as in \cite{Buff:Epstein:PCF, Dujardin:higherbif}. 

We consider families $\mathcal{X}\to V$ of $r$-dimensional varieties $X_{\lambda}\subset \M_d$, parametrized by $\lambda$ in an irreducible quasi-projective variety $V$ of dimension $\ell\ge 1$.   We say that a family $\mathcal{X}\to V$ is {\bf maximally varying} if the natural map from $V$ to the Chow variety $\Ch(\M_d, r, D)$ has finite fibers, where $\Ch(\M_d, r, D)$ is a space of cycles of $\M_d$ with dimension $r$ and degree $\leq D$, with respect to a fixed projective embedding of $\M_d$ defined over $\Q$; see \S\ref{chow}.  

For any positive integer $m$, we let $\cX_V^m$ denote the $m$-th fiber power of $\cX$ over $V$.  We denote elements of $\cX_V^{m}$ by $(\lambda, f_1, \ldots, f_m)$, with $f_i \in X_\lambda$ for each $i$.  Our main theorems are derived from the following result, which is inspired from the Relative Bogomolov Conjecture \cite{RBC:DGH} and by results in \cite{Kuhne:UML} and \cite{Gao:Ge:Kuhne}.

\begin{theorem}\label{relative}
Let $\mathcal{X}\to V$ be a maximally varying family of dimension $r\ge 1$ varieties in $\M_d$, parametrized by an irreducible quasi-projective variety $V$ of dimension $\ell\ge 1$, all defined over $\Qbar$.  Assume that the generic fiber of $\cX\to V$ is geometrically irreducible. There exists $\epsilon=\epsilon(\mathcal{X})>0$ such that the set 
 $$Z_\eps(\cX) := \left\{(\lambda,f_1,\ldots,f_{\ell+1})\in \mathcal{X}_V^{\ell+1}(\Qbar)~:~ \sum_{i=1}^{\ell+1}\hcrit(f_i)<\epsilon \max\{1,h_{\Ch}(\lambda)\}\right\} $$
is not Zariski dense in $\cX_V^{\ell+1}$.
\end{theorem}

To prove Theorem \ref{relative}, we first pass to a maximally varying family of varieties $\cX' \to V$ in a finite branched cover $\cM_d^{cm} \to \M_d$ that parameterizes maps on $\bP^1$ with marked critical points $c_1, \ldots, c_{2d-2}: \cM_d^{cm} \to \bP^1$.  We construct two nonzero bifurcation measures on the fiber power $(\cX')_V^{\ell+1}$, from two distinct $r$-tuples of critical points, and, associated to these measures, we define two adelically metrized line bundles on the quasiprojective $(\cX')_V^{\ell+1}$, in the sense of \cite{Yuan:Zhang:quasiprojective}.  The construction of the measures is carried out in Section \ref{non-degeneracy}, and their nonvanishing ultimately comes from properties of the bifurcation current and its powers in $\M_d$ from  \cite{McMullen:families, D:current, Dujardin:higherbif, Gauthier:Okuyama:Vigny}.  The nonvanishing of the measures implies that the two metrized line bundles are non-degenerate (as defined in \cite{Yuan:Zhang:quasiprojective}).  We are thus able to deduce from the height inequality of \cite[Theorem 1.3.2]{Yuan:Zhang:quasiprojective} that, if $\epsilon >0$ is sufficiently small, then ``most" points of $Z_\epsilon(\cX)$ lie in subvarieties $X_\lambda$ of bounded $h_{\Ch}$-height.

We now proceed by contradiction:  if no such $\epsilon$ exists, then there will be a generic sequence of points in $(\cX')_V^{\ell+1}(\Qbar)$ with  $\sum_{i=1}^{\ell+1}\hcrit(f_i)$ tending to 0.  We apply \cite[Theorem 5.4.3]{Yuan:Zhang:quasiprojective} to deduce that the Galois orbits of these points will be equidistributed with respect the two bifurcation measures, implying the two measures coincide.

The contradiction will come from an analysis of the two measures. More precisely, using a slicing argument as in \cite{Mavraki:Schmidt} (which was revisited in \cite{DM:commonprep}), it suffices to analyze the family of bifurcation measures they induce on the fibers $X_{\la}$. Here we appeal to the maximal variation of the family and the dynamical theory of $J$-stability.  We find a parameter $\lambda_0 \in V$ and conjugacy class $[f_0] \in X_{\lambda_0}$ so that one collection of $r$ critical points is pre-repelling for $f_0$, while one of the critical points in the other $r$-tuple is in the basin of a parabolic point.  We then zoom into the parameter space at $f_0$, where the first $r$-tuple of preperiodic critical points yields an asymptotic similarity between the bifurcation measure and the dynamical space, exactly as in the proof of \cite[Main Theorem]{Buff:Epstein:PCF}.  But zooming in at this rate yields 0 for the bifurcation measure defined with the second $r$-tuple of critical points, because one of the critical points has infinite orbit and is attracted to a parabolic cycle.  Details are provided in Section \ref{proof}.  The proofs of Theorems \ref{main1}, \ref{main2}, \ref{main1Bogomolov}, and \ref{main2Bogomolov} are completed in Section \ref{main proofs}.

\subsection{Questions}
Finally, we remark that our initial (failed) attempts to prove Theorems \ref{relative} were more in line with previously known cases of DAO, combined with the strategy from \cite{Mavraki:Schmidt} to obtain uniform bounds, but we encountered several obstacles. We were led to the following questions, and we view each as a measure-theoretic strengthening of DAO.

\begin{question} \label{q1}
Suppose that $\Lambda$ is a complex manifold of dimension $r\ge 1$, parameterizing a holomorphic family of maps 
	$$f: \Lambda\times \bP^1\to \Lambda\times \bP^1$$ 
defined by $(\lambda, z) \mapsto (\lambda, f_\lambda(z))$, and suppose that $c_1, \ldots, c_{r+1}: \Lambda \to \bP^1$ are marked critical points of $f_\lambda$.  Let $T_i$ denote the bifurcation current on $\Lambda$ associated to the pair $(f, c_i)$, defined in \S\ref{bif}.  Suppose that  
$$T_1 \wedge \cdots \wedge T_{r-1} \wedge T_r \; =\; \alpha \, T_1 \wedge \cdots \wedge T_{r-1} \wedge T_{r+1} \not= 0$$ 
on $\Lambda$ for some $\alpha>0$.  Can we conclude that the pair $(c_r, c_{r+1})$ of critical points is dynamically related along $\Lambda$?  In other words, does there exist a proper, complex-analytic subvariety $X \subset \Lambda \times (\bP^1\times\bP^1)$ which projects surjectively to $\Lambda$, is invariant for the action of $(f,f)$, and contains the graph of $(c_r, c_{r+1})$ over $\Lambda$?  

The question has only been answered in a few special cases with $r=1$:  (1) for the 1-parameter families of polynomials treated by Baker-DeMarco in \cite{BD:polyPCF} parameterized by $\Lambda = \C$, and (2) for the curves $\Lambda = \Per_1(\eta)$ in $\M_2$, for each $\eta \in \C$, of Example \ref{Per1} \cite{DWY:QPer1}. 
\end{question}

\begin{question} \label{q2} 
Suppose that $\Lambda$ is a smooth and irreducible, quasiprojective, complex algebraic variety of dimension $r \geq 2$ and that $f: \Lambda \times \bP^1 \to \Lambda\times\bP^1$ is a morphism that defines an algebraic family of maps on $\bP^1$, with marked critical points.  With the notation of Question \ref{q1}, 
suppose that 
	$$T_1 \wedge \cdots \wedge T_r \not= 0 $$
on $\Lambda$ but 
	$$T_1 \wedge \cdots \wedge T_{r-1} \wedge T_{r+1} = 0$$
on $\Lambda$.  Can we conclude that there exists $i \in \{1, \ldots, r-1\}$ so that the pair $(c_i, c_{r+1})$ is dynamically related along $\Lambda$? (We remark here that one can define a dynamical relation for a tuple $(c_1,\ldots,c_{r+1})$ as for a pair; see \cite[\S 6.2]{D:stableheight}. But it follows by \cite[Proposition 2.21]{Medvedev:Scanlon} that dynamical relations always come in pairs.)  We know of no cases of this question that have been answered.
\end{question}

\begin{remark} 
Versions of Theorems \ref{main1} and \ref{main2} are expected to hold in more general settings, for example when marking points on $\bP^1$ are not necessarily critical, or upon working with preperiodic points in families of maps on $\bP^N$.  See \cite{DM:UDMM}.  Analogues of Questions \ref{q1} and \ref{q2} can also be formulated, but we keep our focus in this article on the postcritically-finite maps on $\bP^1$ and the bifurcation currents associated to marked critical points.
\end{remark}

\subsection{Acknowledgements}
We would like to thank Dori Bejleri, Xavier Buff, Ziyang Gao, Thomas Gauthier, Zhuchao Ji, Lars K\"uhne, Harry Schmidt, Gabriel Vigny, and Junyi Xie for helpful discussions.  We also thank the anonymous referees for their comments and suggestions. Our research was supported in part by the National Science Foundation and by the Natural Sciences and Engineering Research Council of Canada.  In addition, Laura acknowledges the Radcliffe Institute for Advanced Study at Harvard; Myrto acknowledges the University of Toronto and Fields Insitute; and Hexi was partially supported by National Key R\&D Program of China ($2021$YFA$1003202$) and NSFC ($12131016$, $12331004$).

\bigskip
\section{Background on $\M_d$ and its subvarieties}

In this section, we provide important background on the moduli spaces $\M_d$, and we look into properties of a maximally varying family $\cX \to V$ of subvarieties of $\M_d$.  

\subsection{The moduli space $\M_d$ and a critically-marked parameter space} \label{Md}
For each degree $d \geq 2$, the moduli space $\M_d$ is the space of conformal conjugacy classes of maps $f: \bP^1\to\bP^1$ defined over $\C$.  It is an affine algebraic variety of dimension $2d-2$.  Silverman showed that it exists as a geometric quotient over $\Spec \Z$, and he introduced a projective compactification $\Mbar_d$ via the tools of Geometric Invariant Theory which can be defined over $\Q$ \cite{Silverman:Ratd}.  We note that $\M_2 \iso \C^2$ \cite{Milnor:quad} while $\M_d$ is singular in all degrees $d>2$ \cite{West:moduli, Miasnikov:Stout:Williams}.  Because of the existence of maps with automorphisms, this $\M_d$ is not a fine moduli space, but we can pass to a finite branched cover $\cM_d^{cm} \to \M_d$ to build a family 
\begin{equation}\label{universal}
	f : \cM_d^{cm} \times \bP^1 \to \cM_d^{cm} \times \bP^1
\end{equation}
of maps, and we may choose $\cM_d^{cm}$ so that we also have functions 
	$$c_i:  \cM_d^{cm} \to \bP^1$$
marking each of the critical points, $i = 1, \ldots, 2d-2$.

\subsection{PCF maps and Thurston rigidity} \label{Thurston}
A map $f: \bP^1 \to \bP^1$ is {\bf postcritically finite} if each of its critical points has a finite forward orbit.  An important class of examples of such maps are the Latt\`es maps, which are quotients of a self-map of a complex torus; the {\bf flexible Latt\`es maps} are those that are part of a positive-dimensional family in $\M_d$; see \cite{Milnor:Lattes} for more information on Latt\`es maps. We will denote the flexible Latt\`es locus by $\cL_d$ in $\M_d$; it exists only in square degrees. From Thurston's rigidity theorem \cite{Douady:Hubbard:Thurston}, we may deduce that any postcritically finite map which is {\em not} flexible Latt\`es is M\"obius-conjugate to a map $f: \bP^1\to \bP^1$ defined over $\Qbar$.  

\subsection{The bifurcation current and its wedge powers} \label{bif}
Following \cite{Dujardin:Favre:critical}, we can define a {\bf bifurcation current} $T_i$ on $\cM_d^{cm}$ associated to the pair $(f,c_i)$ for the map $f$ of \eqref{universal} and each critical point $c_i$, as follows.  We let $\omega$ denote a Fubini-Study form on $\bP^1$ and let $\hat{\omega}$ be its pullback to $\cM_d^{cm} \times \bP^1$ by the projection map.  Then the sequence of pullbacks $\frac{1}{d^n} (f^n)^* \hat\omega$ converge weakly to a current $\hat{T}_f$ on $\cM_d^{cm} \times \bP^1$.  We let 
\begin{equation} \label{current i}
	T_i = \pi_* \left( \hat{T}_f \wedge [\Gamma_{c_i}] \right),
\end{equation}
where $\pi: \cM_d^{cm} \times \bP^1 \to \cM_d^{cm}$ is the projection and $[\Gamma_{c_i}]$ is the current of integration along the graph of $c_i$.  Then $T_i$ is a positive (1,1)-current with continuous potentials. The sum 
\begin{equation} \label{bif current def}
	\Tbif = \sum_{i = 1}^{2d-2} T_i
\end{equation}
descends to a well-defined positive $(1,1)$-current on the moduli space $\M_d$, and it agrees with the {\bf bifurcation current} introduced in \cite{D:current}.  Because of its continuous potentials, the wedge powers of $\Tbif$ are well defined.  The {\bf bifurcation measure} $\Tbif^{\wedge (2d-2)}$ was first introduced in \cite{Bassanelli:Berteloot} and is non-zero on $\M_d$.  

If $X \subset \cM_d^{cm}$ is an irreducible complex algebraic subvariety of dimension $r > 0$ whose image in $\M_d$ is not contained in the flexible Latt\`es locus $\cL_d$, then $\Tbif \not=0$ on $X$, as a consequence of \cite[Lemma 2.2]{McMullen:families} and \cite[Theorem 1.1]{D:current}.  Moreover, $T_i = 0$ on $X$ if and only if the critical point $c_i$ is persistently preperiodic along $X$ \cite[Theorem 2.5]{Dujardin:Favre:critical}.  It was proved in \cite[Lemma 6.8]{Gauthier:Okuyama:Vigny} that we also have $\Tbif^{\wedge r} \not=0$ on $X$; in other words, there always exists an $r$-tuple of critical points $c_{i_1}, \ldots, c_{i_r}$ so that $T_{i_1} \wedge \cdots \wedge T_{i_r} \not=0$ on $X$.

\subsection{The critical height}
For each $f: \bP^1 \to \bP^1$ of degree $d>1$ defined over $\Qbar$, recall that the canonical height on $\bP^1(\Qbar)$ is defined by 
	$$\hat{h}_f(x) = \lim_{n\to\infty} \frac{1}{d^n} h(f^n(x)),$$
where $h$ is a Weil height on $\bP^1$ \cite{Call:Silverman}.  Counting the critical points with multiplicity, the sum 
	$$\hcrit(f) := \sum_{\{c~: ~f'(c)=0\}} \hat{h}_f(c)$$
defines a function on $\cM_d^{cm}$, which Silverman called the {\bf critical height} and conjectured that it is comparable with a Weil height on $\M_d \setminus\cL_d$; see, for example, \cite{Silverman:moduli}.  This comparability was proved by Ingram in \cite{Ingram:criticalheight}.  

The construction of the critical height was revisited by Yuan and Zhang in \cite[\S 6.3]{Yuan:Zhang:quasiprojective}, where they show that $\hcrit$ is the height associated to a non-degenerate adelically metrized line bundle on $\cM^{cm}_d$ over $\Q$. Indeed, they construct an $f$-invariant adelic line bundle $\overline{L}_f$ on $\cM_d^{cm} \times\bP^1$ for the $f$ of \eqref{universal}. The critical height is then the height function associated to the `restriction' of $\overline{L}_f$ to the ramification divisor of $f$; see \cite[\S 6.3]{Yuan:Zhang:quasiprojective}. Important for us is that the curvature distribution at the archimedean place is equal to the bifurcation current $\Tbif$ defined in \S\ref{bif}. The non-degeneracy of this metrized line bundle follows because the bifurcation measure is non-trivial, as was proved in \cite{Bassanelli:Berteloot}.

Yuan and Zhang use their general equidistribution theorem \cite[Theorem 5.4.3]{Yuan:Zhang:quasiprojective} to give an alternative proof of the arithmetic equidistribution of the points in $\PCF$ to the bifurcation measure $\Tbif^{\wedge 2d-2}$ in \cite[\S 6.3]{Yuan:Zhang:quasiprojective}, a result which was first obtained by Gauthier in \cite{Gauthier:goodheights} (and was proved for the polynomial moduli spaces in \cite{Favre:Gauthier:distribution}).  

\subsection{The Chow variety and maximal variation}
\label{chow}  
For general information about Chow varieties, we refer the reader to \cite{GKZ} \cite{Kollar:rational}.  Here we fix notation and state some basic properties that we need.  We will apply the following discussion to $Y = \M_d$ or $Y = \cM_d^{cm}$.

Suppose that $Y$ is an irreducible quasiprojective algebraic variety, equipped with a projective embedding of $Y \hookrightarrow \bP^N$ defined over $\Q$. We let $\overline{Y}$ denote the closure of $Y$ in $\bP^N$.  For each integer $0 \leq r < \dim Y$, and with respect to this embedding, we first consider the Chow variety $\Ch(\overline{Y}, r, D)$ of algebraic cycles $X \subset \overline{Y}$ in dimension $r$ with degree $\leq D$.  It is an algebraic subvariety, possibly reducible, of the Chow variety $\Ch(\bP^N, r, D)$, consisting of the cycles with dimension $r$ and degree $\leq D$ supported in $\overline{Y}$.  As $\overline{Y}$ is defined over $\Q$, the Chow variety $\Ch(\overline{Y}, r, D)$ is also defined over $\Q$.  We let $\Ch(Y, r, D)$ denote the Zariski-open subset consisting of cycles with support intersecting $Y$.  

We will say that an irreducible subvariety $X \subset Y$ of dimension $r$ has {\bf degree $\leq D$} if its Zariski closure in $\overline{Y}$ is the support of a cycle represented by a point in $\Ch(Y, r, D)$.  Note that the degree depends on the choice of the embedding of $Y$ into $\bP^N$, but the concept of bounded degree does not.  

Suppose that $\pi: \mathcal{X}\to V$ is a family of $r$-dimensional varieties in $Y$, parametrized by $\lambda$ in an irreducible quasi-projective variety $V$ of dimension $\ell\ge 1$. Assume that the generic fiber of $\cX \to V$ is geometrically irreducible. We say that $\mathcal{X}\to V$ is {\bf maximally varying} if the natural map from $V$ to the Chow variety $\Ch(Y, r, D)$ has finite fibers, where we send each irreducible fiber $X_\lambda := \pi^{-1}(\lambda)$ to the point of $\Ch(Y, r, D)$ representing the cycle $[\overline{X}_\lambda]$ on $\overline{Y}$.

For any positive integer $m$, recall that $\cX_V^m$ denotes the $m$-th fiber power of $\cX$ over $V$. There is a natural map 
	$$\rho_m: \mathcal{X}_V^m\to Y^{m}$$ 
defined by $\rho_m(\la,f_1,\ldots,f_{m})=(f_1,\ldots,f_{m})$.

\begin{proposition} \label{max varying}
Suppose that $Y$ is an irreducible quasiprojective complex algebraic variety.  
Suppose that $\mathcal{X}\to V$ is a family of $r$-dimensional varieties $X_{\lambda}\subset Y$, parametrized by $\lambda$ in a quasi-projective variety $V$ of dimension $\ell\ge 1$.  Assume that $V$ and the generic fiber of $\cX \to V$ are irreducible.  If $\cX \to V$ is maximally varying, then the map 
	$$\rho_\ell: \cX^\ell_V \to Y^\ell$$
is generically finite.  In particular $\dim \rho_\ell(\cX^{\ell}_V) = \dim \cX_V^{\ell}$.
\end{proposition}

\begin{proof}
We prove the result by induction on $\ell$.  For $\ell = 1$, the result is clear, as $\dim \cX = r+1$ and maximal variation implies that the image of $\cX$ in $Y$ is not contained in an $r$-dimensional subvariety, so $\dim \rho_1(\cX) = \dim \cX = r+1$.

Now suppose $\ell > 1$ and that we know the result for maximally varying families with base dimension $< \ell$.   Note that $\dim \cX = \ell + r$ and $\dim \cX_V^\ell = \ell(r+1)$.  Choose any smooth point $x_0$ of $\rho_1(\cX)$ which is also smooth for a subvariety $X_{\lambda_0}$ it lies in, and which is regular for the maps $\rho_1: \cX \to \rho_1(\cX)$ and $\pi_1: \rho_\ell(\cX_V^\ell) \to \rho_1(\cX)$ where the latter is the restriction of the projection $Y^\ell \to Y$ to the first coordinate.  Here, `regular for the maps' means that the fiber over $x_0$ has the expected dimension and there is a point in the fiber over $x_0$ mapping surjectively on tangent spaces to that of $x_0$.

Let $Y_1 \subset \rho_1(\cX) \subset Y$ be a subvariety of codimension $r$ in $\rho_1(\cX)$ containing $x_0$, intersecting $X_{\lambda_0}$ (and therefore a general $X_\lambda$) in a finite set.   Consider the projection $p_1:  \cX_V^{\ell} \to \cX$ to the first factor and the projection $p_2: \cX_V^{\ell} \to \cX_V^{\ell-1}$ to the remaining $\ell-1$ factors.  Note that $\rho_1^{-1}(Y_1) \subset \cX$ projects dominantly to the base $V$, and $p_1^{-1}(\rho_1^{-1}(Y_1))$ projects dominantly and generically finitely by $p_2$ to $\cX_V^{\ell-1}$.

Now choose any $Y_1' \subset Y_1$ of codimension 1 containing $x_0$, so of codimension $r+1$ in $\rho_1(\cX)$.  Then $\rho_1^{-1}(Y_1')$ has codimension $r+1$ in $\cX$, and its projection to the base $V$ will be a subvariety $V_1$ of codimension 1.  Let $\cX|_{V_1} \to V_1$ denote the restricted family over the $(\ell-1)$-dimensional base.  Note that $p_1^{-1}(\rho_1^{-1}(Y_1'))$ is contained in $\cX_{V_1}^\ell$ and projects dominantly and generically finitely by $p_2$ to $\cX_{V_1}^{\ell-1}$.  

By induction we know that $\dim \rho_{\ell-1}(\cX_{V_1}^{\ell-1}) = \dim \cX_{V_1}^{\ell-1} = (\ell - 1)(r+1) = \ell(r+1) - (r+1)$.  Therefore 
	$$\dim \rho_\ell(p_1^{-1}(\rho_1^{-1}(Y_1'))) = \dim \cX_{V_1}^{\ell-1} = \ell(r+1)-(r+1).$$ 
Now consider $\pi_1^{-1}(Y_1')$ in $\rho_\ell(\cX_V^\ell)$, which also has codimension $r+1$ by the choice of $x_0$.   But $\pi_1 \circ \rho_\ell = \rho_1\circ p_1$, so that 
	$$\dim \pi_1^{-1}(Y_1') = \dim \rho_\ell(p_1^{-1}(\rho_1^{-1}(Y_1'))),$$
implying that 
	$$\dim \rho_\ell(\cX_V^{\ell}) = \ell(r+1)$$
and the proof is complete. 
\end{proof}

\bigskip
\section{Bifurcation measures on $\cX_V^\ell$ and $\cX_V^{\ell+1}$}\label{non-degeneracy}

Recall the definition of $\cM_d^{cm}$ from \S\ref{Md}.  
Throughout this section, we assume that $\mathcal{X}\to V$ is a maximally varying family of $r$-dimensional varieties $X_{\lambda}\subset \cM_d^{cm}$, parametrized by $\lambda$ in a quasi-projective variety $V$ of dimension $\ell\ge 1$.  Assume that $V$ and the generic fiber of $\cX \to V$ are irreducible.  For any positive integer $m$, recall from \S\ref{strategy} that $\cX_V^m$ denotes the $m$-th fiber power of $\cX$ over $V$, and we denote elements of $\cX_V^{m}$ by $(\lambda, f_1, \ldots, f_m)$, with $f_i \in X_\lambda$ for each $i$. There is a natural map 
\begin{equation} \label{rho}
	\rho_m: \mathcal{X}_V^m\to (\cM^{cm}_d)^m
\end{equation}
defined by $\rho_m(\la,f_1,\ldots,f_{m})=(f_1,\ldots,f_{m})$.  From Proposition \ref{max varying}, we know that $\rho_m$ is generically finite for all $m \geq \ell$. 

In this section we construct nontrivial measures on $\cX_V^{\ell}$ and $\cX_V^{\ell+1}$ that are pullbacks of powers of the bifurcation currents in $\cM_d^{cm}$, defined in \S\ref{bif}.  The measure on $\cX_V^{\ell}$ will be defined in terms of $\Tbif$ of \eqref{bif current def}, while two measures on $\cX_V^{\ell+1}$ will additionally involve choices of $r$-tuples of marked critical points.  The proofs build on the strong rigidity result that the power $\Tbif^{\wedge r}$ of the bifurcation current is nonzero on every $r$-dimensional subvariety of $\M_d\setminus \cL_d$, where $\cL_d$ is the flexible Latt\`es locus; this is proved in \cite[Lemma 6.8]{Gauthier:Okuyama:Vigny} for all $r$.  (See the discussion in \S\ref{bif}.)

We also appeal to descriptions of the support of powers of $\Tbif$ in $\M_d$ in terms of critical orbit behavior obtained by Buff-Epstein in \cite{Buff:Epstein:PCF} and Dujardin in \cite{Dujardin:higherbif}.  Precisely, suppose that $f_1, \ldots, f_m$ is a collection of holomorphic families of maps on $\bP^1$ parameterized by $\lambda$ in a complex manifold $\Lambda$, and let $\Phi := (f_1, \ldots, f_m): \Lambda \times (\bP^1)^m \to \Lambda \times (\bP^1)^m$ be the induced holomorphic map.  We say that an $m$-tuple of marked points ${\bf a} = (a_1, \ldots, a_m): \Lambda \to (\bP^1)^m$ is {\bf transversely} (respectively, {\bf properly}) {\bf prerepelling for $\Phi$ at $\lambda_0$ over $\Lambda$} if there exist repelling periodic points $p_i$ of $f_{\lambda_0}$ 
and integers $n_i \ge 1$ such that 
$f_{\lambda_0}^{n_i}(a_i(\lambda_0)) = p_i$ for each $i = 1, \ldots, m$, 
and the following holds: if 
$p_i: U \to \mathbb{P}^1$ denotes the holomorphic continuation of the repelling periodic point $p_i=p_i(\lambda_0)$ 
over some neighborhood $U \subset \Lambda$ of $\lambda_0$ where it remains repelling, 
then the graphs of the maps
\[
\lambda \longmapsto \big(f_{1,\lambda}^{n_1}(a_1(\lambda)), 
\ldots, f_{m,\lambda}^{n_m}(a_m(\lambda))\big)
\quad \text{and} \quad
\lambda \longmapsto (p_1(\lambda), \ldots, p_m(\lambda))
\]
intersect transversely (respectively, properly) in 
$U \times (\mathbb{P}^1)^m$ at $\lambda_0$.

For positive integers $m \leq 2d-2$, marked critical points ${\bf c} = (c_1, \ldots, c_m)$, and $f_1 = \cdots = f_m := f$, Dujardin proved that there is a dense set of parameters in $\supp T_1 \wedge \cdots \wedge T_m$ in $\cM_d^{cm}$ at which ${\bf c}$ is transversely prerepelling over $\cM_d^{cm}$ \cite[Theorem 1.1]{Dujardin:higherbif}; recall that $T_i$ was defined in \eqref{current i}.  And conversely, if the $m$-tuple of critical points is properly prerepelling for $(f,\ldots, f)$ at some parameter $\lambda_0$ over a parameter space of dimension $\geq m$, then $T_1 \wedge \cdots \wedge T_m$ must be nonzero in that parameter space with $\lambda_0$ in its support, as first observed in \cite{Buff:Epstein:PCF} (where it was proved in the case $m=2d-2$ and assuming transversely prerepelling) and revisited in \cite[Theorem 6.2]{Gauthier:strong}, \cite[Theorem 2.3]{Gauthier:abscont}, and \cite[Proposition 3.7]{Berteloot:Bianchi:Dupont} in more general contexts and without assuming transversality.  See \cite[Proposition 4.8]{DM:commonprep} for the version we use, which implies that the latter converse result remains true with any $m$-tuple of maps $(f_1,\ldots, f_m)$ of degree $d$ in place of $(f,\ldots,f)$.

\subsection{Measure on $\cX_V^\ell$}

\begin{proposition}\label{bi-bif}
Let $\mathcal{X}\to V$ be a maximally varying family of dimension $r\ge 1$ varieties in $\cM_d^{cm}$ with $\dim V = \ell \geq 1$, and let $\rho_{\ell}$ be the map of \eqref{rho} for $m = \ell$.  Set $\cY_\ell = \rho_\ell(\cX_V^\ell)$.  Let $\Tbif$ denote the bifurcation current in $\cM_d^{cm}$ and let $\pi_i: (\cM_d^{cm})^{\ell}\to \cM_d^{cm}$ be the projection to the $i$-th factor for $i=1,\ldots,\ell$. Then 
$$\pi_1^{*}(T_{\mathrm{bif}}^{\wedge r+1})\wedge \cdots \wedge \pi_{\ell}^{*}(T_{\mathrm{bif}}^{\wedge r+1})\wedge [\cY_\ell]\neq 0$$
in $(\cM_d^{cm})^\ell$, so that 
	$$\mu_\ell := \rho_\ell^* \left(\pi_1^{*}(T_{\mathrm{bif}}^{\wedge r+1})\wedge \cdots \wedge \pi_{\ell}^{*}(T_{\mathrm{bif}}^{\wedge r+1})\right)$$
defines a positive measure on $\cX_V^\ell$.  Moreover, there is a collection of marked critical points 
	$${\bf c} = (c_{1,1}, \ldots, c_{1, r+1}, \ldots, c_{\ell, 1}, \ldots, c_{\ell, r+1})$$
defined on $\cX_V^\ell$ which is transversely prerepelling for the family of product maps 
	$$\Phi := (f_1,\ldots, f_1, \ldots,f_\ell, \ldots, f_\ell)$$ 
at a parameter $x_0$ over $\cX_V^\ell$, where each $f_i$ is repeated $r+1$ times.
\end{proposition}

If ${\bf c}$ satisfies the conclusion of Proposition \ref{bi-bif} for $\Phi$, we say that it {\bf witnesses} the nonvanishing of $\mu_\ell$.  

\begin{proof}
Let $\cY \subset \cM_d^{cm}$ be the image of $\rho_1$, and let $k$ be the dimension of $\cY$, so that 
	$$2 \leq r+1 \leq k  \leq r+ \ell = \dim \cX.$$  
Since the projection $\cM_d^{cm} \to \M_d$ has finite fibers, we know from \cite[Lemma 6.8]{Gauthier:Okuyama:Vigny} that $\Tbif^{\wedge k} \not=0$ on $\cY$. 

If $\ell=1$, then we are done, since $k = r+1$ and $\Tbif^{\wedge (r+1)} \wedge [\cY]$ in $\cM_d^{cm}$.

Assume $\ell > 1$.  Let
	$$\Phi = (f_1,\ldots, f_1, \ldots, f_\ell, \ldots, f_\ell)$$
denote the family of product maps on $(\bP^1)^{\ell(r+1)}$ parameterized by $\cY_\ell \subset (\cM_d^{cm})^\ell$, where each coordinate $f_i$ is repeated $r+1$ times.

Let $U\subset \cY$ be a Zariski-open set so that the fibers of the two maps $\rho_1: \cX \to \cY$ and $\pi_1: \cY_\ell \to \cY$ to the first factor of $(\cM_d^{cm})^\ell$ have the expected dimension over $U$, namely $r+\ell-k$ for fibers of $\rho_1$ and $\ell + r\ell - k$ for $\pi_1$, and the maps are surjective on tangent spaces at some point of each fiber over elements of $U$.  Note that, by the symmetry of $\cY_\ell$, the fibers will have the same properties for projections $\pi_i$ for each $i$.  From \cite[Theorem 1.1]{Dujardin:higherbif}, there exists a $k$-tuple of critical points $\{c_1, \ldots, c_k\}$ and a parameter $y_0 \in U$ in the support of $\Tbif^{\wedge k}$ at which these critical points are transversely prerepelling for the family of maps $f_1$ parameterized by $U$. We may assume that $y_0$ is a smooth point of $\cY$ and of the subvariety $X_{\lambda_0}$ for a smooth parameter $\lambda_0 \in V$.

Now choose a collection of $r+1$ critical points $c_{1,1}, \ldots, c_{1,r+1} \in \{c_1, \ldots, c_k\}$ so that $(c_{1,1}, \ldots, c_{1,r})$ is transversely prerepelling at $y_0$ over $X_{\lambda_0}$ and the full $(r+1)$-tuple is transversely prerepelling at $y_0$ over $U$. Let $Z_1 \subset U$ be an irreducible component of the codimension-$(r+1)$ subvariety containing $y_0$ where these critical points are persistently preperiodic.  Note that $y_0$ is an isolated point of the intersection of $Z_1$ with $X_{\lambda_0}$.  Moreover, we know that $\pi_1^{-1}(Z_1)$ has codimension $r+1$ in $\cY_\ell$.

Note that the preimage $\rho_\ell^{-1}(\pi_1^{-1}(Z_1))$ in $\cX_V^{\ell}$ projects to a subvariety $V_1$ of codimension 1 in the base $V$.  In fact it maps with finite fibers to the $(\ell-1)$-th fiber power of a subfamily 
	$$\cX|_{V_1} \to V_1$$ 
where $\dim V_1 = \ell -1$. Now let $\cY_1$ be the image of this subfamily $\cX|_{V_1}$ in $\cM_d^{cm}$.  We repeat the argument with $\cY_1$ in place of $\cY$.  By maximal variation of $\cX$, we know that 
	$$2 \leq r+1 \leq k_1 := \dim \cY_1 \leq r+ \ell -1$$
and $\Tbif^{\wedge k_1} \not=0$ on $\cY_1$.  We choose another collection of critical points $c_{2,1}, \ldots, c_{2,r+1} \in \{c_1, \ldots, c_k\}$ so that the first $r$ are transversely prerepelling at $y_0$ over $X_{\lambda_0}$ and the full $(r+1)$-tuple is transversely prerepelling at $y_0$ over $\cY_1$.  Note that we could take $c_{2,j} = c_{1,j}$ for $j = 1, \ldots, r$.  We let $Z_2 \subset \cY_1$ be the codimension-$(r+1)$ subvariety containing $y_0$ where these critical points are persistently preperiodic.  Then $\pi_2^{-1}(Z_2) \cap \pi_1^{-1}(Z_1)$ has codimension $2(r+1)$ in $\cY_\ell$.  

In this way, we inductively construct a parameter $y = (y_0, \ldots, y_0) \in \cY_\ell$ and an $\ell(r+1)$-tuple of marked critical points 
	$${\bf c} = (c_{1,1}, \ldots, c_{1, r+1}, \ldots, c_{\ell, 1}, \ldots, c_{\ell, r+1})$$
which is transversely prerepelling for the map $\Phi$ at $y$ over $\cY_\ell$.  In particular, in the language of \cite[\S4.3]{DM:commonprep}, the graph $\Gamma_{\bf c}$ of ${\bf c}$ in $\cY_\ell \times (\bP^1)^{\ell(r+1)}$ defines a rigid repeller for the map $\Phi$ at this parameter, and we conclude from \cite[Proposition 4.8]{DM:commonprep} that 
\begin{equation} \label{witness}
	\hat{T}_\Phi^{\wedge \ell(r+1)} \wedge [\Gamma_{\bf c}] \not=0,
\end{equation}
where 
	$$\hat{T}_\Phi = \sum_{i=1}^\ell \sum_{j = 1}^{r+1} p_{i,j}^* \hat{T}_{f_i}$$ 
is the dynamical Green current for $\Phi$ on $\Lambda \times (\bP^1)^{\ell(r+1)}$, for the projections $p_{i,j}: \Lambda \times (\bP^1)^{\ell(r+1)} \to \Lambda \times \bP^1$.  Unwinding the definitions, we conclude that 
	$$\pi_1^{*}(T_{f_1, c_{1,1}}\wedge\cdots \wedge T_{f_1, c_{1, r+1}}) \wedge \cdots \wedge \pi_{\ell}^{*}(T_{f_\ell, c_{\ell,1}} \wedge \cdots \wedge T_{f_\ell, c_{\ell, r+1}})\wedge [\cY_\ell]\neq 0$$ 
and therefore that 
	$$\pi_1^{*}(T_{\mathrm{bif}}^{\wedge r+1})\wedge \cdots \wedge \pi_{\ell}^{*}(T_{\mathrm{bif}}^{\wedge r+1})\wedge [\cY_\ell]\neq 0$$
so that the measure $\mu_\ell \not=0$.  Here we used the fact that $\dim \cX_V^{\ell}=\dim\cY_{\ell}$ from the maximal variation.

Finally, since the map from $\cX_V^{\ell}$ to $\cY_\ell$ is generically finite by Proposition \ref{max varying}, our initial choice of $y_0$ guarantees that the marked critical points ${\bf c}$ are also transversely prerepelling at a preimage of $y$, over $\cX_V^\ell$. 
\end{proof}

\subsection{Measures on $\cX_V^{\ell+1}$}
Let $\mu_\ell$ be the measure of Proposition \ref{bi-bif}, witnessed by the marked critical point ${\bf c}$.  We now use these critical points to construct nonvanishing measures on the next fiber power $\cX_V^{\ell+1}$.  Let 
	$$\Phi = (f_1, \ldots, f_1, \ldots, f_\ell, \ldots, f_\ell)$$
be as in Proposition \ref{bi-bif} parameterized by $\cX_V^\ell$ via the map $\rho_{\ell}$ of \eqref{rho}, and let 
	$$\Phi_{\ell+1} = (f_1, \ldots, f_1, \ldots, f_\ell, \ldots, f_\ell, f_{\ell+1}, \ldots, f_{\ell+1})$$
be the family of maps on $(\bP^1)^{\ell(r+1)+r}$ parameterized by $\cX_V^{\ell+1}$, where $f_1, \ldots, f_{\ell}$ are each repeated $r+1$ times, but the final coordinate $f_{\ell+1}$ is only repeated $r$ times.  Specializing at each $\lambda\in V$ we also have a family of maps 
$$F^{(\lambda)}:=(f,\ldots,f): X_{\lambda}\times (\bP^1)^r\to X_{\lambda}\times (\bP^1)^r,$$
parametrized by $X_{\lambda} \subset \cM_d^{cm}$, and, focusing on the $\ell$-th factor of $\cX_V^\ell$ and associated components of ${\bf c}$, we will use the marked critical points 
	$$c_{\ell,1}|_{X_\lambda}, \ldots, c_{\ell,r+1}|_{X_\lambda}: X_{\lambda}\to \bP^1$$ 
to define two measures $\mu_\ell^i$ and $\mu_\ell^j$ on $\cX_V^{\ell+1}$.  Conditions (1) and (2) of Proposition \ref{two measures} will allow us to conclude that $\mu_\ell^i$ and $\mu_\ell^j$ are nonzero.  Conditions (3) and (4) will allow us to prove in Section \ref{proof} that $\mu_\ell^i$ and $\mu_\ell^j$ are not proportional measures.

\begin{proposition}\label{two measures} 
Suppose the marked critical point 
	$${\bf c} = (c_{1,1}, \ldots, c_{1, r+1}, \ldots, c_{\ell, 1}, \ldots, c_{\ell, r+1})$$ 
witnesses $\mu_\ell \not=0$ on $\cX_V^\ell$.  Then there are parameters $x_0 = (\lambda_0, t_1, \ldots, t_\ell)$ and $x_1 = (\lambda_1, s_1, \ldots, s_\ell)$ in the support of $\mu_\ell$ in $\cX_V^\ell$ and indices $i \not= j$ in $\{1, \ldots, r+1\}$ so that  
\begin{enumerate}
\item the $r$-tuple $(c_{\ell,1}, \ldots, \widehat{c_{\ell, i}}, \ldots c_{\ell, r+1})$ is transversely prerepelling for $F^{(\lambda_0)}$ at $t_\ell$ over $X_{\lambda_0}$; 
\item the $r$-tuple $(c_{\ell,1}, \ldots, \widehat{c_{\ell, j}}, \ldots c_{\ell, r+1})$ is properly prerepelling for $F^{(\lambda_0)}$ at $t_\ell$ over $X_{\lambda_0}$;
\item the $r$-tuple $(c_{\ell,1}, \ldots, \widehat{c_{\ell, i}}, \ldots c_{\ell, r+1})$ is transversely prerepelling for $F^{(\lambda_1)}$ at $s_\ell$ over $X_{\lambda_1}$; and 
\item the critical point $c_{\ell, i}$ is attracted to a parabolic periodic point of $f$ at $s_\ell \in  X_{\lambda_1}$,
\end{enumerate}
where $\widehat{c_{\ell, i}}$ means that the $i$-th coordinate is omitted.  Let $\Pi_1:\mathcal{X}^{\ell}_V \times_V \mathcal{X}\to \cX_V^\ell$ and $\Pi_2:\mathcal{X}_V^{\ell}\times_V\mathcal{X}\to\mathcal{X}$. Then
$$\mu_\ell^i := \Pi_1^{*}(\mu_{\ell}) \wedge \Pi_2^*(T_{c_{\ell,1}}\wedge \cdots  \widehat{T_{c_{\ell, i}}}  \cdots \wedge T_{c_{\ell,r+1}}) \neq 0$$
and
	$$\mu_\ell^j := \Pi_1^{*}(\mu_{\ell}) \wedge \Pi_2^*(T_{c_{\ell,1}}\wedge \cdots  \widehat{T_{c_{\ell, j}}}  \cdots \wedge T_{c_{\ell,r+1}}) \neq 0$$
on $\cX_V^{\ell+1}$. 
\end{proposition}

\begin{proof}
Via the map $\rho_\ell: \cX_V^{\ell} \to (\cM_d^{cm})^{\ell}$, we view $\cX_V^{\ell}$ as the parameter space for a family of product maps $\Phi = (f_1, \ldots, f_1, \ldots, f_\ell, \ldots, f_\ell)$ on $(\bP^1)^{\ell(r+1)}$. From Proposition \ref{bi-bif}, we know that ${\bf c}$ is transversely prerepelling for $\Phi$ at a parameter $x_0 = (\lambda_0, t_1, \ldots, t_\ell) \in \supp \mu_\ell$.

Let $S_{i,j}$ be the irreducible component of the hypersurface in $\cX_V^\ell$ containing $x_0$ along which $c_{i,j}$ is persistently preperiodic for $f_i$, $i = 1, \ldots, \ell$, $j = 1, \ldots, r+1$.  Consider the component of the intersection 
	$$S_{1,1} \cap \cdots \cap S_{\ell-1, r+1}$$
containing $x_0$ in $\cX_V^\ell$.  Similar to the proof of Proposition \ref{max varying}, for dimension reasons, it maps (generically finitely) to a family $\cX_1 \to V_1$ of $r$-dimensional varieties over a 1-dimensional base, where these marked critical points are persistently preperiodic for $f_1, \ldots, f_{\ell-1}$.  We will continue to write $x_0 = (\lambda_0, t_\ell)$ for the point we started with, but now viewed in $\cX_1$.  

By the transversality, we know that $r$ of the remaining critical points, say $c_{\ell,1}, \ldots, c_{\ell, r}$, are transversely prerepelling at $t_\ell$ for $F^{(\lambda_0)} = (f, f, \ldots, f)$ over $X_{\lambda_0}$.

We now aim to show that by modifying the parameter $x_0$ if necessary, we can arrange so that another $r$-tuple, say $c_{\ell, 2}, \ldots, c_{\ell, r+1}$, is also properly prerepelling at $t_\ell$ for $F^{(\lambda_0)}$ over $X_{\lambda_0}$. Note that transversality of the intersection $S_{1,1} \cap \cdots \cap S_{\ell-1, r+1}$ at the original $x_0$ persists in a neighborhood, so the first $(\ell-1)(r+1)$ critical points remain transversely prerepelling in a neighborhood of $x_0$ for $(f_1, \ldots, f_{\ell-1})$.  

Consider the intersection 
	$$S_{\ell,1} \cap \cdots \cap S_{\ell, r}$$
in $\cX_1$.  These hypersurfaces intersect transversely at $x_0$ and so define a smooth holomorphic curve $\Gamma$ containing $x_0$ which is locally a section of $\cX_1 \to V_1$ near $\lambda_0$.  The intersection of $\Gamma$ with $X_{\lambda_0}$ is $x_0 = (\lambda_0, t_\ell)$.  Note further that $c_{\ell, r+1}$ is not persistently preperiodic along $\Gamma$ (because of the original setup of transverse intersections), and so $c_{\ell, r+1}$ must be active along $\Gamma$ with $x_0$ in its bifurcation locus (because it lands at a repelling cycle there). The bifurcation locus for $c_{\ell,r+1}$ in $\Gamma$ has no isolated points, so there must be an infinite collection of parameters $x \in \Gamma$ where $c_{\ell, r+1}$ is preperiodic to a repelling cycle for $f_\ell$, dense in the bifurcation locus of $c_{\ell, r+1}$ in a neighborhood of $x_0$.  

Now omit $S_{\ell, 1}$ from the intersection, so that $S_{\ell,2} \cap \cdots \cap S_{\ell, r}$ (in $\mathcal{X}_1$) defines a 1-parameter family of algebraic curves $\cC$, for all $\lambda$ in a small neighborhood of $\lambda_0\in V_1$, by
 	$$C_\lambda = (S_{\ell,2} \cap \cdots \cap S_{\ell, r}) \cap X_\lambda.$$
We know that $S_{\ell, r+1}$ intersects the total space of this family $\cC$ transversely at the given point $x_0$.  Suppose that it fails to intersect properly when restricted to $X_{\lambda_0}$; this would mean that $c_{\ell, r+1}$ is persistently preperiodic along the curve $C_{\lambda_0}$ in $X_{\lambda_0}$. Suppose this happens at all parameters $x \in \Gamma$ where $c_{\ell, r+1}$ is preperiodic to a repelling cycle.  This implies that the associated bifurcation current $T_{c_{\ell, r+1}}$ vanishes identically on each of these curves $C_x$ passing through each such $x$ in the family $\cC$.  By continuity of potentials of the bifurcation currents, we deduce that $T_{c_{\ell, r+1}}$ vanishes along $C_x$ for all $x$ in $\supp T_{c_{\ell, r+1}}|_\Gamma$ near $x_0$.  But this means that $c_{\ell, r+1}$ is persistently preperiodic along each of these curves \cite[Theorem 2.5]{Dujardin:Favre:critical}.  As they form an uncountable family, we see that $c_{\ell, r+1}$ must be preperiodic along the entire family $\cC$, which is a contradiction. 

We conclude that we can adjust our parameter $\lambda_0$ slightly, if needed, so that points $c_{\ell, 1}, \ldots, c_{\ell, r}$ are transversely prerepelling at a new choice of $x_0 \in \Gamma$ along $X_{\lambda_0}$ while $c_{\ell, 2}, \ldots, c_{\ell, r+1}$ are properly prerepelling at $x_0$ along $X_{\lambda_0}$.  This completes the proofs of (1) and (2) of the proposition.

As $c_{\ell,r+1}$  is bifurcating along the curve we called $\Gamma$,  the characterizations of $J$-stability (as in \cite{Mane:Sad:Sullivan}) imply that there are parameters, everywhere dense in its bifurcation locus, for which the maps have a parabolic cycle; the theory of renormalization (see \cite{Douady:Hubbard:polynomial-like}) implies that $c_{\ell, r+1}$ will lie in the basin of the parabolic cycle.  We let $\lambda_1$ be a parameter close to $\lambda_0$ in the support of $\pi_* \mu_\ell$, where $\pi$ is the projection to $V$, so that $c_{\ell,r+1}$ is attracted to a parabolic cycle at $x_1 := \Gamma \cap X_{\lambda_1}$, while $c_{\ell,1}, \ldots, c_{\ell,r}$ remain transversely prerepelling at $x_1$ along $X_{\lambda_1}$.  This completes the proof of properties (1)--(4).

To prove the non-vanishing of $\mu_\ell^i$ and $\mu_\ell^j$, consider the map $\Phi_{\ell+1} = f_1^{\times (r+1)} \times \cdots \times f_{\ell}^{\times (r+1)}\times f_{\ell+1}^{\times r}$ on $(\bP^1)^{\ell(r+1) + r}$ over $\cX_V^{\ell+1}$.  The argument above, leading to properties (1) and (2), shows that we can label critical points of $f_1, \ldots, f_{\ell+1}$ so that the graphs of the marked points
	$${\bf c}_i = (c_{1,1}, \ldots, c_{1, r+1}, \ldots, c_{\ell, 1}, \ldots, c_{\ell, r+1}, c_{\ell+1,1}, \ldots \widehat{c_{\ell+1, i}} \ldots, c_{\ell+1, r+1})$$
and 
	$${\bf c}_j = (c_{1,1}, \ldots, c_{1, r+1}, \ldots, c_{\ell, 1}, \ldots, c_{\ell, r+1}, c_{\ell+1,1}, \ldots \widehat{c_{\ell+1, j}} \ldots, c_{\ell+1, r+1})$$
each define a rigid repeller for the map $\Phi_{\ell+1}$ over a neighborhood of a parameter $(\lambda_0, t_1, \ldots, t_\ell, t_\ell)$ in $\cX_V^{\ell+1}$, in the sense of \cite[\S4.3]{DM:commonprep}, and so the measures are nonzero by \cite[Proposition 4.8]{DM:commonprep}.  
\end{proof}

\bigskip
\section{Proof of Theorem \ref{relative}} \label{proof}

Suppose that $\mathcal{X}\to V$ is a maximally varying family of $r$-dimensional varieties $X_{\lambda}\subset M_d$ parametrized by $V$ of dimension $\ell\ge 1$, defined over $\Qbar$. 
As in \S\ref{Md}, we may pass to a finite branched cover $\cM_d^{cm} \to \M_d$ over which we can define a family of maps $f:\bP^1 \to \bP^1$ with $2d-2$ marked critical points.  Pulling back each $X_\lambda$ and possibly replacing $V$ with a finite branched cover, we will continue to use the notation $\cX \to V$ for the family of varieties in $\cM_d^{cm}$.  Therefore, we may assume that we can define the family of product maps 
	$$\Phi= (f_1,\ldots, f_1, \ldots, f_\ell, \ldots, f_\ell),$$
on $(\bP^1)^{(r+1)\ell}$, parametrized by $\mathcal{X}^{\ell}_V$ and defined over $\Qbar$, where each coordinate $f_i$ is repeated $r+1$ times and a family 
$$\Phi_{(\ell+1)}=(f_{\ell+1},\ldots,f_{\ell+1})$$ 
on $(\mathbb{P}^1)^{r}$, parametrized by $\mathcal{X}$, where each coordinate is repeated $r$ times. 
We denote by $\overline{L}_{\Phi}$  and $\overline{L}_{\Phi_{(\ell+1)}}$ the $\Phi$ (respectively $\Phi_{(\ell+1)}$)-invariant extension of the classical polarization as defined by  \cite[Theorem 6.1.1]{Yuan:Zhang:quasiprojective}. In particular, using the notation in \cite{Yuan:Zhang:quasiprojective}, we have that $\overline{L}_{\Phi_{(\ell+1)}}\in \widehat{\mathrm{Pic}}(\mathcal{X}\times (\mathbb{P}^1)^r)_{\bQ,\mathrm{nef}}$ and $\overline{L}_{\Phi}\in \widehat{\mathrm{Pic}}(\mathcal{X}^{\ell}_V\times (\mathbb{P}^1)^{\dim \mathcal{X}^{\ell}_V})_{\bQ,\mathrm{nef}}$. 

In what follows it is convenient to work with the Deligne pairing (of relative dimension dimension $0$), of the invariant metrized line bundles restricted to graphs of critical marked points (see \cite[\S 4.2]{Yuan:Zhang:quasiprojective} for metrics of the Deligne pairing and \cite[\S 6.2.2]{Yuan:Zhang:quasiprojective} for the `restriction' of the invariant metrized line bundle to a subvariety).  By \cite[Theorem 4.2.3]{Yuan:Zhang:quasiprojective}, to each marked tuple of critical points ${\bf c}: \mathcal{X}^{\ell}_V\to (\mathbb{P}^1)^{\dim \mathcal{X}^{\ell}_V}$ we can associate a metrized line bundle $\< \overline{L}_{\Phi}|_{\Gamma_{\bf c}}\>\in \widehat{\mathrm{Pic}}( \mathcal{X}^{\ell}_V)_{\bQ,\mathrm{nef}}$, where $\Gamma_{\bf c}$ denotes the graph of ${\bf c}$ in $\mathcal{X}^{\ell}_V\times (\mathbb{P}^1)^{\dim \mathcal{X}^{\ell}_V}$. Moreover, we let 
$$\overline{L}:=\otimes _{{\bf c}} \<\overline{L}_{\Phi}|_{\Gamma_{\bf c}}\>\in \widehat{\mathrm{Pic}}( \mathcal{X}^{\ell}_V)_{\bQ,\mathrm{nef}}$$ be the tensor product over all marked tuples of critical points ${\bf c}$ for the maps $f_1, \ldots, f_\ell$.

Now let $\mathbf{c}=(c_{1,1},\ldots,c_{\ell,r+1}):\mathcal{X}^{\ell}_{V}\to (\bP^1)^{(r+1)\ell}$ be a marked critical point for $\Phi$ that witnesses $\mu_{\ell} \not=0$, as in Proposition \ref{bi-bif}.
Reordering the critical points if needed, we can find $x_0,x_1\in \supp\mu_{\ell}$ as in Proposition \ref{two measures}, with $i=r+1$ and $j=1$.
We write 
$c_i:=c_{\ell,i}:\mathcal{X}\to \mathbb{P}^1$ in what follows. 
As before, using \cite[Theorem 4.2.3]{Yuan:Zhang:quasiprojective}, we define two adelic metrized line bundles on $\mathcal{X}$ by 
\begin{align}\label{bif1}
\begin{split}
\overline{L}_1&:=\<\overline{L}_{\Phi_{(\ell+1)}}|_{\Gamma_{(c_1,\ldots,c_r)}}\>\in \widehat{\mathrm{Pic}}( \mathcal{X})_{\bQ,\mathrm{nef}}\\
\overline{L}_2&:=\<\overline{L}_{\Phi_{(\ell+1)}}|_{\Gamma_{(c_2,\ldots,c_{r+1})}}\>\in \widehat{\mathrm{Pic}}( \mathcal{X})_{\bQ,\mathrm{nef}},
\end{split}
\end{align}
and note that their curvature forms (at an archimedean place) are the bifurcation currents
\begin{align}\label{bif2}
\begin{split}
c_1(\overline{L}_1)&=T_{f_{\ell+1},c_1}+\cdots+T_{f_{\ell+1},c_r}\\
c_1(\overline{L}_2)&=T_{f_{\ell+1},c_2}+\cdots+T_{f_{\ell+1}, c_{r+1}}.
\end{split}
\end{align}
We can now define two metrized line bundles 
$$\overline{M}_i=\Pi_1^{*}(\overline{L})+\Pi_2^{*}(\overline{L}_{i})\in \widehat{\mathrm{Pic}}( \mathcal{X}^{\ell+1}_V)_{\bQ,\mathrm{nef}}$$ on $\mathcal{X}^{\ell+1}_V$ for $i=1,2$. By \cite[Lemma 6.2.1, Theorem 4.1.3]{Yuan:Zhang:quasiprojective}, the heights associated to $\overline{M}_i$ vanish at PCF points; see also \cite[Lemma 6.3.2]{Yuan:Zhang:quasiprojective}. More precisely, for a point $(\lambda_0,t_1,\ldots,t_{\ell+1})\in \mathcal{X}^{\ell+1}_V(\Qbar)$ we have 
\begin{align}
\begin{split}
h_{\overline{M}_1}(\lambda_0,t_1,\ldots,t_{\ell+1})&= \sum_{i=1}^{\ell}h_{\mathrm{crit}}(f_{i}) + \sum_{j=1}^{r}\hat{h}_{f_{\lambda_0,t_{\ell+1}}}(c_{j}(\lambda_0,t_{\ell+1}))\\
h_{\overline{M}_2}(\lambda_0,t_1,\ldots,t_{\ell+1})&= \sum_{i=1}^{\ell}h_{\mathrm{crit}}(f_{i}) + \sum_{j=2}^{r+1}\hat{h}_{f_{\lambda_0,t_{\ell+1}}}(c_{j}(\lambda_0,t_{\ell+1})).
\end{split}
\end{align}
We will now show that $\overline{M}_i$ is ``non-degenerate'' for $i=1,2$ in the terminology of \cite[\S 6.2.2]{Yuan:Zhang:quasiprojective}. Equivalently, by  \cite[Lemma 5.4]{Yuan:Zhang:quasiprojective}, it suffices to show that 
$$c_1(\overline{M_i})^{\wedge \ell(r+1)+r}\neq 0.$$
Recall that $T_{f_i,c_j}^{\wedge 2}=0$ for all $i,j$ and similarly $(\Pi^*_2c_1(\overline{L}_i))^{\wedge r+1}= (\Pi_2)^*(c_1(\overline{L}_i)^{\wedge r+1})=0$ for $i=1,2$, by the definition of the bifurcation current as limits of pullbacks, because the map $\Phi_{(\ell+1)}$ acts on the $r$-dimensional $(\mathbb{P}^1)^{r}$. Moreover, since $\dim\mathcal{X}^{\ell}_V= \ell(r+1)$ we also have that $\Pi_1^{*}(c_1(\overline{L}))^{\wedge \ell(r+1)+1}=0$.
Thus , using also  \cite[Theorem 4.2.3]{Yuan:Zhang:quasiprojective} as before, we have 
\begin{align}
c_1(\<\overline{L}\>)^{\wedge (r+1)\ell}=\al_1\,\mu_{\ell}
\end{align}
on $\cX_V^\ell$, and 
\begin{align}
\begin{split}
c_1(\overline{M_1})^{\wedge \ell(r+1)+r}=\al_2\, \Pi_1^{*}(\mu_{\ell})\wedge \Pi_2^{*}(T_{f_{\ell+1},c_1}\wedge\cdots\wedge T_{f_{\ell+1},c_r})=\al_2\,\mu_{\ell}^{r+1}\neq 0,\\
c_1(\overline{M_2})^{\wedge \ell(r+1)+r}=\al_3\,\Pi_1^{*}(\mu_{\ell})\wedge \Pi_2^{*}(T_{f_{\ell+1},c_2}\wedge\cdots\wedge T_{f_{\ell+1},c_{r+1}})=\al_3\,\mu^1_{\ell}\neq 0,
\end{split}
\end{align}
for positive constants $\al_1,\al_2,\al_3>0$. Their non-vanishing follows by Proposition \ref{two measures} (and our choice of critical points). 
Therefore, by the height inequality \cite[Theorem 1.3.2]{Yuan:Zhang:quasiprojective} we infer that there are positive constants $\epsilon_1=\epsilon_1(\mathcal{X})>0,~\epsilon_2=\epsilon_2(\mathcal{X})>0$
such that 
\begin{align}\label{big height gap}
\sum_{i=1}^{\ell+1}\hat{h}_{\mathrm{crit}}(f_i)\ge \epsilon_1 h_{\Ch}(\lambda)-\epsilon_2,
\end{align}
for all $(\lambda,f_1,\ldots,f_{\ell+1})\in U(\Qbar)$ for a Zariski open and dense subset $U\subset \mathcal{X}^{\ell+1}_V$. 

It remains to show that for some $\epsilon=\epsilon(\mathcal{X})$, the set $$\left\{(\lambda,f_1,\ldots,f_{\ell+1})\in \mathcal{X}_V^{\ell+1}(\Qbar)~:~ \sum_{i=1}^{\ell+1}\hcrit(f_i)<\epsilon\right\} $$
is not Zariski dense in $\cX_V^{\ell+1}$. 
We prove this result by contradiction. Suppose that there is a generic Galois invariant sequence $\{x_n\}_n$ of critically small points in $\cX_V^{\ell+1}$, so that $h_{\overline{M}_i}(x_n)\to 0$ for $i=1,2$. Since $\overline{M}_i$ is nef and non-degenerate, Yuan-Zhang's fundamental inequality \cite[Theorem 5.3.3]{Yuan:Zhang:quasiprojective} implies that $h_{\overline{M}_i}(\cX_V^{\ell+1})=0$ for $i=1,2$.  We can use \cite[Theorem 5.4.3]{Yuan:Zhang:quasiprojective} with the Galois-invariant critically small set to get that 
$$\Pi_1^{*}(\mu_{\ell})\wedge \Pi_2^{*}(T_{f_{\ell+1},c_1}\wedge\cdots\wedge T_{f_{\ell+1},c_r})= \alpha \, \Pi_1^{*}(\mu_{\ell})\wedge \Pi_2^{*}(T_{f_{\ell+1},c_2}\wedge\cdots\wedge T_{f_{\ell+1},c_{r+1}})\neq 0,$$
on $\cX_V^{\ell+1}$ for some constant $\alpha>0$.

We now slice these measures with respect to the projection $\cX_V^{\ell +1} \to \cX_V^\ell$; see, for example \cite[Proposition 4.3]{Bassanelli:Berteloot}.
This implies that 
\begin{equation}\label{equality slice}
\int_{\mathcal{X}^{\ell}_V}\left( \int_{X_{\pi(x)}}\phi \cdot  \bigwedge_{i=1}^r T_{f_{\ell+1},c_i}|_{X_{\pi(x)}}\right)  d\mu_{\ell}(x) = \al \int_{\mathcal{X}^{\ell}_V}\left( \int_{X_{\pi(x)}}\phi \cdot  \bigwedge_{i=2}^{r+1} T_{f_{\ell+1},c_i}|_{X_{\pi(x)}}\right)  d\mu_{\ell}(x)
\end{equation}
for every continuous and compactly supported function $\phi$ on $\mathcal{X}^{\ell+1}_V$.  Here, $\pi: \mathcal{X}^{\ell}_V\to V$ denotes the projection to the base.  (Strictly speaking, since the measure $\mu_\ell$ is not smooth, we should approximate $\mu_\ell$ by smooth and compactly-supported forms on $\cX_V^\ell$ to deduce \eqref{equality slice}; compare the proof of \cite[Theorem 1.8]{Mavraki:Schmidt}.)

For each $\lambda \in V$, set
$$\mu^{r+1}_{\lambda}:=T_{f_{\ell+1},c_1}\wedge\cdots\wedge T_{f_{\ell+1},c_r}|_{X_{\lambda}}$$
and 
$$\mu_\lambda^1=T_{f_{\ell+1},c_2}\wedge\cdots\wedge T_{f_{\ell+1},c_{r+1}}|_{X_{\lambda}}.$$

We infer from \eqref{equality slice} that 
\begin{align}\label{equality for all}
\mu^{r+1}_{\la}=\al \, \mu^{1}_{\la}
\end{align}
as measures on $X_{\la}$, for some constant $\alpha>0$ and for all $\lambda$ in the support of $\pi_*\mu_\ell$ in $V$. Indeed, suppose there exists $b \in V(\C)$ where \eqref{equality for all} doesn't hold, and so that $\pi_{*}\mu_{\ell}(U) >0$ for every open neighborhood $U$ of $b$.  Then we can find a continuous, non-negative, compactly supported function $\psi$ on $\mathcal{X}$ such that 
$$\int_{X_b} \psi \, d\mu^{r+1}_{b}(x) \neq \al \int_{X_b} \psi \, d\mu^{1}_{b}(x).$$  
By weak continuity of the measures $\mu^{r+1}_{\lambda}$ and $\mu_\lambda^1$ as $\lambda$ varies (since the bifurcation currents have continuous potentials), we find that 
$$\int_{X_{\la}} \psi \, d\mu^{r+1}_{\la}(x) \neq \al\int_{X_{\la}} \psi \, d\mu^{1}_{\la}(x)$$ 
for all $\la$ in a small neighborhood $U$ of $b$. Therefore, setting 
$$\phi(\la, t_1, \ldots, t_{\ell},t_{\ell+1}) = h_b(\la) \psi(\la,t_{\ell+1})$$ on $\mathcal{X}^{\ell+1}_V$ for a non-negative function $h_b$ supported in a small neighborhood of $b$, the equality \eqref{equality slice} will fail.

The remainder of the proof is devoted to deriving a contradiction from the equality \eqref{equality for all}.  We first sketch the idea. 
Recall from Proposition \ref{two measures} that we have parameter $x_1 = (\lambda_1, s_1, \ldots, s_\ell) \in \supp \mu_\ell$ so that the tuple of marked critical points $(c_1, \ldots, c_r)$ is transversely prerepelling for $F^{(\lambda_1)} = (f,\ldots, f)$ at $s_\ell$ over $X_{\lambda_1}$. This implies that $\mu_{\lambda_1}^{r+1}$ is nonzero on $X_{\lambda_1}$, with $s_\ell$ in its support by \cite{Buff:Epstein:PCF}.  Moreover, as in \cite[Section 3]{Buff:Epstein:PCF} and the generalization appearing in \cite[Proposition 3.2]{Gauthier:abscont}, we know that if we zoom in on the measure $\mu^{r+1}_{\lambda_1}$ on $X_{\lambda_1}$ at the parameter $s_\ell$ at an exponential rate in $n$ (determined by the multipliers of the repelling cycles in the forward orbits of the marked critical points), and rescale the measure appropriately we get a nontrivial limit; see \eqref{zoominprerepelling}. On the other hand, we know from Proposition \ref{two measures} that the critical point $c_{r+1}$ is attracted to a parabolic cycle at $s_{\ell}\in X_{\lambda_1}$.  Zooming into the measure $\mu_{\lambda_1}^1$ on $X_{\lambda_1}$ at $s_\ell$ at the same exponential rate (which is supposed to give a nonzero limit), we actually get 0 from the following lemma. This will produce our desired contradiction. 

\begin{lemma} \label{r zoom}
Suppose that $f: \bD^r \times \bP^1\to \bD^r \times \bP^1$ is a holomorphic family of maps with degree $d\ge 2$ parametrized by $\bD^r \subset \bC^r$ with a marked critical point $c: \bD^r \to \bP^1$.  Assume that $c(0)$ is attracted to a parabolic fixed point at $z=0$ for $f_0$.  Then for any constant $\rho > 1$, there exists $\delta>0$ so that the sequence of functions 
	$$\alpha_n(x) := f_{x/\rho^n}^n(c(x/\rho^n))$$
converges uniformly to 0 on the polydisk $(\bD_\delta)^r$, as $n\to\infty$.
\end{lemma}

\proof
Let $\rho>1$. Since $|f'_0(0)| = 1$, we can first fix a small $\eps>0$ and $\rho_\eps>1$ so that 
	$$\left|f_x'(z)\right| \leq \rho_\eps < \rho,$$
for all $x\in \bD^r_{\epsilon}$ and $z\in \bD_\eps$.  We replace $c$ with an iterate so that $c(0)$ and all its forward iterates under $f_0$ are in the disk $\bD_{\eps/2}$.  Choose $M>0$ so that the gradient of $f$ (with respect to parameter $x$) satisfies
   $$\left|\nabla_x f_x(z) \right| \leq M$$
for all  $x\in \bD^r$ and $z\in \bD_\eps$.  Note that
   $$\alpha_n(x)=f_{x/\rho^n}(\alpha_{n-1}(x/\rho)). $$
Taking derivatives, we have
\begin{eqnarray} \label{deriv}
|\nabla \alpha_n(x)|&=&\left|\frac{1}{\rho^n}\cdot (\nabla_w f_w) (\alpha_{n-1}(x/\rho))+\frac{1}{\rho}  f_{x/\rho^n}'(\alpha_{n-1}(x/\rho)) \nabla_{x/\rho} \alpha_{n-1}(x/\rho) \right| \nonumber \\
    &\leq& \frac{M}{\rho^n} +\frac{\rho_\eps}{\rho}\, |\nabla_{x/\rho}\alpha_{n-1}(x/\rho)|
 \end{eqnarray}
as long as $\alpha_{n-1}(x/\rho)$ lies in $\bD_\eps$.
 
Now choose $n_0$ so that 
   $$\frac{M}{\rho^{n_0}}<\frac{\rho_\eps}{\rho}.$$ 
We take a very small $\delta_1$ so that 
   $$\alpha_{n_0-1}(x)\in \bD_{\eps}$$
for all $x \in \bD_{\delta_1}^r$, and set 
   $$M_1:=\max_{x\in \bD_{\delta_1}^r}|\nabla\alpha_{n_0-1}(x)|.$$
From the inequality \eqref{deriv}, we have
   $$|\nabla\alpha_{n_0}(x)|  \;\leq \; \frac{M}{\rho^{n_0}} + \frac{\rho_\eps}{\rho}\, |\nabla_{x/\rho}\alpha_{n_0-1}(x/\rho)| \;\leq \;  \frac{\rho_\eps}{\rho}+ \frac{\rho_\eps}{\rho} M_1$$
for each $x\in \bD_{\delta_1}^r$. Let 
	$$M_2:= \frac{\rho_\eps}{\rho}+ \frac{\rho_\eps}{\rho} M_1,$$ 
and define $M_i$ recursively by
   $$M_{i+1}=\frac{\rho_\eps}{\rho^i} + \frac{\rho_\eps}{\rho}M_i$$
for $i\geq 2$.  Note that $M_i \to 0$ as $i\to \infty$, because 
	$$M_i = \left(\frac{\rho_\eps}{\rho}\right)^{i-1}(M_1+1) + \frac{1}{\rho^i} \sum_{j=1}^{i-2} \rho_\eps^j.$$

Finally, we can show inductively that our derivative estimate \eqref{deriv} holds when applied to $\alpha_n$ for all $n > n_0$ on a sufficiently small disk near 0.  This will imply that 
   $$|\nabla\alpha_{n_0+i}(x)|<\frac{M}{\rho^{n_0+i}}+\frac{\rho_\eps}{\rho}\, |\nabla\alpha_{n_0+i-1}(x/\rho)| \leq     \frac{\rho_\eps}{\rho^{i+1}}+\frac{\rho_\eps}{\rho}M_{i+1} = M_{i+2}$$
for all $x$ in a small disk around 0.  We know it holds for $i=0$ for all $|x| < \delta_1$.  We take $M_0:=\max_{i\geq 2} M_i$, and set 
    $$\delta:=\min (\delta_1, \eps/(2M_0), \eps).$$
 As $\alpha _{n_0}(0)\in \bD_{\eps/2}$, the inequality \eqref{deriv} gives $|\nabla \alpha_{n_0}(x)|<M_2\leq M_0$, so that $\alpha_{n_0}(x)$ is in $\bD_{\eps}$ for each  $x\in \bD^r_{\delta}$. By induction, we see that if $\alpha_{n_0+i-1}(x)$ is in $\bD_{\eps}$ with $|\nabla\alpha_{n_0+i-1}|\leq M_{i+1}$, then $\alpha_{n_0+i}(x)$ is in $\bD_{\eps}$, because $\alpha_{n_0+i}(0)$ is in $\bD_{\eps/2}$ and $\delta<\eps/(2M_0)$. Since $M_i\to 0$ as $i$ tends to infinity, we see that $\alpha_n$ converges to $0$ uniformly on  $\bD^r_\delta$. 
\qed

\bigskip
We now return to our proof.  Abusing the notation slightly, we write $\{f_s\}_{s\in X_{\lambda_1}}$ to denote the restriction of the family of maps to $X_{\lambda_1}$, $c_{r+1}$ for the restriction of $c_{r+1}$ to $X_{\lambda_1}$, and $T_{c_{r+1}}$ its associated bifurcation current on $X_{\lambda_1}$. We know from Lemma \ref{r zoom} that, zooming into $s_\ell$ in $X_{\lambda_1}$ at any exponential rate, the graphs of iterates of $c_{r+1}$ are converging to the graph of a constant function.  

Let $p_0\in\bP^1$ be the parabolic periodic point of $f_{s_\ell}$; replacing $f_{s_\ell}$ with an iterate, we have $f_{s_\ell}^n(c_{r+1}(s_\ell))\to p_0$ as $n\to\infty$.  After a change of coordinates if necessary, we may assume that $c_{r+1}(s_\ell)$ and $p_0$ are not $\infty$. 
Recalling the definition of the bifurcation current associated to $c_{r+1}$, we can work in homogeneous coordinates on $\bP^1$ over a neighborhood of $s_\ell$ in $X_{\lambda_1}$. We define
	$$G_s(X,Y) := \lim_{i\to \infty}\frac{1}{d^i}\log\|F_{s}^{i}(X,Y)\|,$$
with a choice of homogeneous lift $F_{s}$  for $f_{s}$ normalized so that $G_{s_\ell}(p_0,1)= 0$. Note that $G_s(c_{r+1}(s), 1)$ is  a potential  of the bifurcation current $T_{c_{r+1}}$ on a neighborhood of $s_\ell\in X_{\lambda_1}$. As long as $f^n_s(c_{r+1}(s))$ is not $\infty$, which is always true when $s$ is in $ \mathbb{D}(s_\ell, \delta/\rho^n)$ for small $\delta>0$, any $\rho>1$ and big $n$ by Lemma \ref{r zoom}, we define
   $$U_n(s):=G_s(f^n_s(c_{r+1}(s)),1)$$
so that 
	$$dd^c U_n(s)=d^n\cdot dd^c G_s(c_{r+1}(s), 1) = d^n \, T_{c_{r+1}}$$
where defined, as $G_s(F^n_s(c_{r+1}(s),1)) = d^n\cdot G_s(c_{r+1}(s),1)$ differs from $U_n(s)$ by a harmonic function.

Let $r_0$ be any biholomorphic map from $\bD^{r}$ to a neighborhood of $s_\ell\in X_{\lambda_1}$, with $r_0(0)=s_\ell$.  From Lemma \ref{r zoom} and the choice of our lift, we deduce that, for any $\rho>1$, the functions $U_n\circ r_0(y/\rho^n)$ converge uniformly to $G_{s_{\ell}}(p_0,1)=0$  on a small neighborhood of $(0, \ldots, 0)\in \bD^{r}$. As a consequence,  for any $\rho_i\in \C$ with $|\rho_i|>1$ ,  we can pick some $\rho>1$ with $\rho< \min_{1\leq i\leq {r}} (|\rho_i|)$ so that the sequence of functions 
\begin{equation}\label{zoomed seq}
   U_n\circ r_0(y_1/\rho_1^n, \ldots, y_{r}/\rho_{r}^n)=U_n\circ r_0(A_n(y)/\rho^n) \to 0
\end{equation}
uniformly, as $A_n(y):=(y_1\cdot\rho^n/\rho_1^n, \ldots, y_{r}\cdot\rho^n/\rho_{r}^n)$ converges to zero uniformly for $n\to \infty$.  

Following the notation of \cite[Proposition 3.2]{Gauthier:abscont}, we now let $r_n: \bD^r \to X_{\lambda_1}$ be the sequence of maps that define the zooms near the parameter $s_\ell$ for the measure $\mu^{r+1}_{\lambda_1}$ corresponding to the transversely prerepelling critical point at $x_1$, as described before Lemma \ref{r zoom}. From \cite[Proposition 3.2]{Gauthier:abscont} we know that there is a measure $\mu\neq 0$ (which can be described explicitly) such that
\begin{equation}\label{zoominprerepelling}
d^{nr}\,(r_n)^{*}(\mu^{r+1}_{\lambda_1})\to \mu,
\end{equation}  
uniformly as $n\to\infty$. On the other hand, it follows from \eqref{zoomed seq} that the functions $U_n\circ r_n$ also converge uniformly to 0. We conclude that $d^n \, (r_n)^{*} T_{c_{r+1}} \to 0$ weakly, and therefore 
	$$d^{nr} \, (r_n)^{*} \mu^1_{\lambda_1} \to 0.$$
Therefore, $\mu^1_{\lambda_1}$ cannot be proportional to $\mu^{r+1}_{\lambda_1}$, contradicting \eqref{equality for all}. This completes the proof of Theorem \ref{relative}.

\begin{remark}
The proof would simplify if we knew that the $(r+1)$-th critical point $c_{r+1}$ were preperiodic to a parabolic cycle, instead of only in the basin of that cycle. We could employ ideas that appeared in \cite[Proposition 5.2]{Ghioca:Nguyen:Ye:DMM1} (written for measures in the dynamical plane) or the proof of \cite[Proposition 5.12]{Favre:Gauthier:book} (for $1$-parameter families of polynomials and marked points). But the critical point $c_{r+1}$ is sometimes the only ``free" critical point and must therefore have an infinite orbit in the presence of a parabolic cycle.  
\end{remark}


\bigskip
\section{Proofs of Theorems \ref{main1}, \ref{main2}, \ref{main1Bogomolov}, and \ref{main2Bogomolov}}
\label{main proofs}

An elementary, but important, geometric observation is the following \cite[Lemma 4.3]{Gao:Ge:Kuhne}:  If $\Sigma$ is a subset of a complex projective algebraic variety $X$, equipped with an ample line bundle $L$, and if the product $\Sigma^M$ lies in a proper subvariety $Y \subset X^M$ for some integer $M\geq 1$, then $\Sigma$ must be contained in a proper subvariety $Z$ of $X$, where the number of components of $Z$ and the degree of each  are bounded by constants depending only on $M$, $\dim X$, the degree of $X$, and the sum of the degrees of the components of $Y$. Here, we measure degree in $X^M$ with respect to the ample line bundle $p_1^*L \otimes \cdots \otimes p_M^*L$ on $X^M$, where $p_i: X^M \to X$ is the projection to the $i$-th factor.

 Fix degree $d\geq 2$. For positive integers $r$ and $D$, we will apply Theorem \ref{relative} to families $\cX \to V$ over a subvariety $V \subset \Ch(\M_d, r,D)$, all defined over $\Qbar$.  We work inductively on the dimension $r$ of the fibers in the family.  

First suppose that $\cC \to V$ is a family of curves in $\M_d$ of degree $\leq D$, with irreducible $V \subset \Ch(\M_d, 1, D)$ defined over $\Qbar$ and of dimension $\ell\geq 1$, and assume that the generic fiber of $\cC \to V$ is geometrically irreducible.  From Theorem \ref{relative} applied to the fiber power $\cC_V^{\ell+1}$, we know that there exists an $\eps = \eps(\cC)$ so that points $(\lambda, f_1, \ldots, f_{\ell+1}) \in \cC_V^{\ell+1}(\Qbar)$ with $\sum_i \hcrit(f_i) < \eps\max\{1,h_{\Ch}(\lambda)\}$ are not Zariski dense.  We let $Z_\eps$ denote the Zariski closure of this set in $\cC_V^{\ell+1}$, over $\Qbar$.  It follows that there is a Zariski open $U \subset V$ (defined over $\Qbar$) so that $Z_\eps \cap C_\lambda^{\ell+1}$ is a proper algebraic subvariety of $(C_\lambda)^{\ell+1}$ for all $\lambda \in U(\C)$, with its number of irreducible components bounded over $\lambda \in U(\C)$. It is worth noting that all points $(\lambda, f_1, \ldots, f_{\ell+1}) \in \cC_V^{\ell+1}(\C)$ where each $f_i$ is in $\M_d(\Qbar)$ with $\sum_i \hcrit(f_i) < \eps$ must lie in $Z_\eps$, also for $\lambda \in V(\C) \setminus V(\Qbar)$; indeed, any such $(\ell+1)$-tuple must lie in a family of curves in $V$ defined over $\Qbar$ that includes this $C_\lambda$.

We now apply \cite[Lemma 4.3]{Gao:Ge:Kuhne} to the set 
	$$\Sigma_{\lambda,\eps} := \left\{f \in C_\lambda \cap \M_d(\Qbar): \hcrit(f) <\frac {\eps}{\ell+1}\max\{1,h_{\Ch}(\la)\}\right\} \subset C_\lambda$$
for $\lambda \in U(\Qbar)$, and to 
	$$\Sigma_{\lambda,\eps} := \left\{f \in C_\lambda \cap \M_d(\Qbar): \hcrit(f) <\frac {\eps}{\ell+1}\right\} \subset C_\lambda$$
for $\lambda \in U(\C)\setminus U(\Qbar)$, so that $\Sigma_{\lambda,\eps}^{\ell+1}\subset Z_{\eps}\cap C^{\ell+1}_{\lambda}$ for all $\lambda \in U(\C)$. 
We deduce the existence of a constant $N$ so that $|\Sigma_{\lambda,\eps}| \leq N$ for all $\lambda \in U(\C)$. Note here that we have used the fact that the number of components of $Z_{\eps}\cap C^{\ell+1}_{\la}$ and the degree of each is uniformly bounded for $\lambda\in U(\C)$ (controlled by the geometry of its generic fiber). Note also that if $0\le \eps'<\eps$, then $\Sigma_{\lambda,\eps'}\subset \Sigma_{\lambda,\eps}$ and hence $|\Sigma_{\lambda,\eps'}|\leq N$ for each $\lambda \in U(\C)$. 

 We then repeat the argument with the families of curves over the positive-dimensional components in the complement of $V\setminus U$; we work with irreducible components so that the generic fibers of these families will be (geometrically) irreducible.  Note that each family we treat is defined over $\Qbar$.  We increase the bound $N$ and decrease $\epsilon$ as needed. We are left with a finite set $S$ in $V(\Qbar)$ containing all of the irreducible PCF-special curves with degree $\leq D$ as well as all curves in $V$ with infinitely many points of critical height $< \frac{\eps}{\ell+1}\max\{1,h_{\Ch}(\la)\}$ (or $< \frac{\eps}{\ell+1}$ over complex parameters).
This proves Theorems \ref{main2} and \ref{main2Bogomolov} for the curves in $V$. If for any of the curves $C_\lambda$ in this leftover finite set $S \subset V(\Qbar)$, we have $N<|\Sigma_{\lambda,\eps}|<\infty$, we just increase the bound $N$ as needed so that Theorem \ref{main1Bogomolov} and \ref{main2} hold for all curves in $V$. 
We repeat the argument for all irreducible components $V$ of $\Ch(\M_d, 1, D)$ with geometrically irreducible generic fiber and adjust the constants $\eps,~N$ as needed. This proves Theorems \ref{main1}, \ref{main2}, \ref{main1Bogomolov} and \ref{main2Bogomolov} for curves. 

Now fix $r>1$ and suppose we know the conclusion of Theorems \ref{main1}, \ref{main2}, \ref{main1Bogomolov}, and \ref{main2Bogomolov} for subvarieties in $\M_d$ of dimension $<r$.  Consider any family $\cX \to V$ of $r$-dimensional varieties in $\M_d$ over an irreducible and quasiprojective $V \subset \Ch(\M_d, r, D)$ defined over $\Qbar$ and of dimension $\ell\geq 1$ for which the generic fiber of $\cX \to V$ is geometrically irreducible.

As in the case of curves, we infer from Theorem \ref{relative} that there exists $\eps = \eps(\cX)>0$ so that all points $(\lambda, f_1, \ldots, f_{\ell+1}) \in \cX_V^{\ell+1}(\Qbar)$ where each $f_i$ is in $\M_d(\Qbar)$ with $\sum_i \hcrit(f_i) < \eps\max\{1,h_{\Ch}(\la)\}$ lie in a proper closed subvariety $Z_\eps$ in $\cX_V^{\ell+1}$. Therefore, as before, from \cite[Lemma 4.3]{Gao:Ge:Kuhne}, there is a Zariski-open subset $U\subset V$ (defined over $\Qbar$) and constants $N_1$ and $B_1$ so that the points of 
	$$\Sigma^r_{\lambda,\eps} := \left\{f \in X_\lambda \cap \M_d(\Qbar): \hcrit(f) < \frac{\eps}{\ell+1}\max\{1,h_{\Ch}(f)\}\right\} \subset X_\lambda$$
for $\lambda \in U(\Qbar)$, and of 
	$$\Sigma^r_{\lambda,\eps} := \left\{f \in X_\lambda \cap \M_d(\Qbar): \hcrit(f) < \frac{\eps}{\ell+1}\right\} \subset X_\lambda$$
for $\lambda \in U(\C)\setminus U(\Qbar)$, lie in a proper algebraic subvariety $Z_{(\lambda)} \subset X_\lambda$ which has $\leq N_1$ irreducible components, each with degree $\leq B_1$, for all $\lambda \in U(\C)$.  Note that $\Sigma^r_{\lambda,\epsilon}$ is not necessarily Zariski dense in $Z_{(\lambda)}$, so we apply the induction hypotheses (specifically, the constants from each of the statements of Theorems \ref{main1}, \ref{main2}, \ref{main1Bogomolov}, \ref{main2Bogomolov} for subvarieties of dimension $< r$) to each irreducible component $Y_{(\lambda)}$ of $Z_{(\lambda)}$, for all $\lambda \in U(\C)$.  Thus, by reducing $\eps$ if needed, we conclude that the Zariski-closure of $\Sigma^r_{\lambda,\eps}\cap Y_{(\la)}$ has uniformly bounded number of components, each with uniformly bounded degree for all $\lambda \in U(\C)$ . Further, adjusting these uniform bounds if needed, using the induction hypothesis, we can arrange that they also bound the geometry of the Zariski closure of $\Sigma^r_{\lambda, \eps'}\cap Y_{(\la)}\subset\Sigma^r_{\lambda, \eps}\cap Y_{(\la)}$ for all $0\le \eps'<\eps$ (and in particular of $Y_{(\lambda)} \cap \PCF$).  Combined with the initial bound of $N_1$ on the number of components of $Z_{(\lambda)}$, repeating this procedure  provides uniform bounds on the Zariski closures of $\Sigma^r_{\lambda,\eps'}$ for all $\eps'<\eps$ and of $\overline{X_\lambda \cap \PCF}$ for all $\lambda \in U(\C)$.

We now repeat the argument over each of the irreducible components of the complement $V \setminus U$ corresponding to families with geometrically irreducible generic fibers, and we note that each family we treat is defined over $\Qbar$. Finally, we are left with a finite set $S$ in $V(\Qbar)$ consisting of $r$-dimensional varieties that are either PCF-special or contain a Zariski-dense set of points in $\Sigma^r_{\lambda, \eps}$.  We repeat the arguments above for curves, reducing $\eps$ and increasing $B$ and $N$ as needed, once again appealing to the induction hypotheses.  We have thus completed the proofs of Theorems \ref{main1}, \ref{main2}, \ref{main1Bogomolov}, and \ref{main2Bogomolov}.

\bigskip \bigskip

\begin{thebibliography}{DKY2}

\bibitem[BD]{BD:polyPCF}
Matthew Baker and Laura DeMarco.
\newblock {Special curves and postcritically-finite polynomials}.
\newblock {\em Forum Math. Pi} {\bf 1}(2013), 35 pages.

\bibitem[BB]{Bassanelli:Berteloot}
Giovanni Bassanelli and Fran\c{c}ois Berteloot.
\newblock {Bifurcation currents in holomorphic dynamics on {$\Bbb P^k$}}.
\newblock {\em J. Reine Angew. Math.} {\bf 608}(2007), 201--235.

\bibitem[BBD]{Berteloot:Bianchi:Dupont}
Fran\c{c}ois Berteloot, Fabrizio Bianchi, and Christophe Dupont.
\newblock {Dynamical stability and {L}yapunov exponents for holomorphic
  endomorphisms of {$\Bbb P^k$}}.
\newblock {\em Ann. Sci. \'{E}c. Norm. Sup\'{e}r. (4)} {\bf 51}(2018),
  215--262.

\bibitem[BE]{Buff:Epstein:PCF}
Xavier Buff and Adam Epstein.
\newblock {Bifurcation measure and postcritically finite rational maps}.
\newblock In {\em Complex dynamics}, pages 491--512. A K Peters, Wellesley, MA,
  2009.

\bibitem[CS]{Call:Silverman}
Gregory~S. Call and Joseph~H. Silverman.
\newblock {Canonical heights on varieties with morphisms}.
\newblock {\em Compositio Math.} {\bf 89}(1993), 163--205.

\bibitem[DP]{David:Philippon:2007}
Sinnou David and Patrice Philippon.
\newblock {Minorations des hauteurs normalis{\'e}es des sous-vari\'et\'es des
  puissances des courbes elliptiques}.
\newblock {\em Int. Math. Res. Pap.} (2007).

\bibitem[De1]{D:current}
Laura DeMarco.
\newblock {Dynamics of rational maps: a current on the bifurcation locus}.
\newblock {\em Math. Res. Lett.} {\bf 8}(2001), 57--66.

\bibitem[De2]{D:stableheight}
Laura DeMarco.
\newblock {Bifurcations, intersections, and heights}.
\newblock {\em Algebra Number Theory.} {\bf 10}(2016), 1031--1056.

\bibitem[DKY1]{DKY:UMM}
Laura DeMarco, Holly Krieger, and Hexi Ye.
\newblock {Uniform {M}anin-{M}umford for a family of genus 2 curves}.
\newblock {\em Ann. of Math. (2)} {\bf 191}(2020), 949--1001.

\bibitem[DKY2]{DKY:quad}
Laura DeMarco, Holly Krieger, and Hexi Ye.
\newblock {Common preperiodic points for quadratic polynomials}.
\newblock {\em J. Mod. Dyn.} {\bf 18}(2022), 363--413.

\bibitem[DM1]{DM:commonprep}
Laura DeMarco and Niki~Myrto Mavraki.
\newblock {Dynamics on $\mathbb{P}^1$: preperiodic points and pairwise
  stability}.
\newblock{Compos. Math.}, {\bf 160}, 2024, no.2, 356--387.

\bibitem[DM2]{DM:Mandelbrot}
Laura DeMarco and Niki~Myrto Mavraki.
\newblock {Geometry of PCF parameters in spaces of quadratic polynomials}.
\newblock{Algebra Number Theory}, {\bf 19}, 2025, no.11, 2163--2183.
  

\bibitem[DM3]{DM:UDMM}
Laura DeMarco and Niki~Myrto Mavraki.
\newblock {The geometry of preperiodic points in families of maps on
  $\mathbb{P}^N$}.
\newblock {\em Preprint, {\em arXiv:2407.10894 [math.DS]}}.

\bibitem[DWY]{DWY:QPer1}
Laura DeMarco, Xiaoguang Wang, and Hexi Ye.
\newblock {Bifurcation measures and quadratic rational maps}.
\newblock {\em Proc. London Math. Soc.} {\bf 111}(2015), 149--180.

\bibitem[DGH]{DGH:uniformity}
Vesselin Dimitrov, Ziyang Gao, and Philipp Habegger.
\newblock {Uniformity in {M}ordell-{L}ang for curves}.
\newblock {\em Ann. of Math. (2)} {\bf 194}(2021), 237--298.

\bibitem[DGH2]{RBC:DGH}
Vesselin Dimitrov,  Ziyang Gao, and Philipp Habegger.
\newblock{A consequence of the relative {B}ogomolov conjecture}.
\newblock{\em J. Number Theory},
{\bf 230}, 2022, 146--160.

\bibitem[DH1]{Douady:Hubbard:polynomial-like}
Adrien Douady and John Hamal Hubbard.
\newblock {On the dynamics of polynomial-like mappings}.
\newblock {\em Ann. Sci. \'Ecole Norm. Sup. (4)} {\bf 18}(1985), 287--343.

\bibitem[DH]2{Douady:Hubbard:Thurston}
Adrien Douady and John H.~Hubbard.
\newblock {A proof of {T}hurston's topological characterization of rational
  functions}.
\newblock {\em Acta Math.} {\bf 171}(1993), 263--297.

\bibitem[Du]{Dujardin:higherbif}
Romain Dujardin.
\newblock {The supports of higher bifurcation currents}.
\newblock {\em Ann. Fac. Sci. Toulouse Math. (6)} {\bf 22}(2013), 445--464.

\bibitem[DF]{Dujardin:Favre:critical}
Romain Dujardin and Charles Favre.
\newblock {Distribution of rational maps with a preperiodic critical point}.
\newblock {\em Amer. J. Math.} {\bf 130}(2008), 979--1032.

\bibitem[FG1]{Favre:Gauthier:distribution}
C.~Favre and T.~Gauthier.
\newblock {Distribution of postcritically finite polynomials}.
\newblock {\em Israel J. Math.} {\bf 209}(2015), 235--292.

\bibitem[FG2]{Favre:Gauthier:cubics}
Charles Favre and Thomas Gauthier.
\newblock {Classification of special curves in the space of cubic polynomials}.
\newblock {\em Int. Math. Res. Not. IMRN} {\bf 2}(2018), 362--411.

\bibitem[FG3]{Favre:Gauthier:book}
Charles Favre and Thomas Gauthier.
\newblock {\em The Arithmetic of Polynomial Dynamical Pairs}, volume 214 of
  {\em Annals of Mathematics Studies}.
\newblock Princeton University Press, Princeton, NJ, 2022.

\bibitem[Gao]{Gao:survey}
Ziyang Gao.
\newblock {Recent developments of the Uniform Mordell-Lang Conjecture}.
\newblock {\em Preprint, {\em arXiv:2104.03431v5 [math.NT]}}.

\bibitem[GGK]{Gao:Ge:Kuhne}
Ziyang Gao, Tangli Ge, and Lars K{\"u}hne.
\newblock {The Uniform Mordell-Lang Conjecture}.
\newblock {\em Preprint, {\em arXiv:2105.15085v2 [math.NT]}}.

\bibitem[GH]{Gao:Habegger:RMM}
Ziyang Gao and Philipp Habegger.
\newblock {The Relative Manin-Mumford Conjecture}.
\newblock {\em Preprint, {\em arXiv:2303.05045v2 [math.NT]}}.

\bibitem[Ga1]{Gauthier:strong}
Thomas Gauthier.
\newblock {Strong bifurcation loci of full {H}ausdorff dimension}.
\newblock {\em Ann. Sci. \'Ec. Norm. Sup\'er. (4)} {\bf 45}(2012), 947--984
  (2013).

\bibitem[Ga2]{Gauthier:abscont}
Thomas Gauthier.
\newblock {Dynamical pairs with an absolutely continuous bifurcation measure}.
\newblock{Ann. Fac. Sci. Toulouse Math. (6)}{\bf 32}(2023), no.2, 203--230.

\bibitem[Ga3]{Gauthier:goodheights}
Thomas Gauthier.
\newblock {Good height functions on quasiprojective varieties: equidistribution
  and applications in dynamics}.
\newblock {\em Preprint, {\em arXiv:2105.02479v3 [math.NT]}}.

\bibitem[GOV]{Gauthier:Okuyama:Vigny}
Thomas Gauthier, Y\^{u}suke Okuyama, and Gabriel Vigny.
\newblock {Approximation of non-archimedean {L}yapunov exponents and
  applications over global fields}.
\newblock {\em Trans. Amer. Math. Soc.} {\bf 373}(2020), 8963--9011.

\bibitem[GTV]{Gauthier:Taflin:Vigny}
Thomas Gauthier, Johan Taflin, and Gabriel Vigny.
\newblock {Sparsity of postcritically finite maps of $\mathbb{P}^k$ and beyond:
  a complex analytic approach}.
\newblock {\em Preprint, {\em arXiv:2305.02246 [math.DS]}}.

\bibitem[GKZ]{GKZ}
I.M. Gelfand, M.M. Kapranov, and A.V. Zelevinsky.
\newblock {\em Discriminants, resultants, and multidimensional determinants}.
\newblock Mathematics: Theory \& Applications. Birkh\"auser Boston Inc.,
  Boston, MA, 1994.

\bibitem[GHT]{Ghioca:Hsia:Tucker}
D.~Ghioca, L.-C. Hsia, and T.~Tucker.
\newblock {Preperiodic points for families of polynomials}.
\newblock {\em Algebra Number Theory} {\bf 7}(2013), 701--732.

\bibitem[GNY]{Ghioca:Nguyen:Ye:DMM1}
D.~Ghioca, K.~D. Nguyen, and H.~Ye.
\newblock {The dynamical {M}anin-{M}umford conjecture and the dynamical
  {B}ogomolov conjecture for split rational maps}.
\newblock {\em J. Eur. Math. Soc. (JEMS)} {\bf 21}(2019), 1571--1594.

\bibitem[GY]{Ghioca:Ye:cubics}
Dragos Ghioca and Hexi Ye.
\newblock {A dynamical variant of the {A}ndr\'{e}-{O}ort conjecture}.
\newblock {\em Int. Math. Res. Not. IMRN} {\bf 8}(2018), 2447--2480.

\bibitem[In]{Ingram:criticalheight}
Patrick Ingram.
\newblock {The critical height is a moduli height}.
\newblock {\em Duke Math. J.} {\bf 167}(2018), 1311--1346.

\bibitem[JX]{Ji:Xie:DAO}
Zhuchao Ji and Junyi Xie.
\newblock {DAO for curves}.
\newblock {\em Preprint, {\em arXiv:2302.02583v2 [math.DS]}}.

\bibitem[Ko]{Kollar:rational}
J\'{a}nos Koll\'{a}r.
\newblock {\em Rational curves on algebraic varieties}, volume~32 of {\em
  Ergebnisse der Mathematik und ihrer Grenzgebiete. 3. Folge. A Series of
  Modern Surveys in Mathematics [Results in Mathematics and Related Areas. 3rd
  Series. A Series of Modern Surveys in Mathematics]}.
\newblock Springer-Verlag, Berlin, 1996.

\bibitem[K{\"u}]{Kuhne:UML}
Lars K{\"u}hne.
\newblock {Equidistribution in families of abelian varieties and uniformity}.
\newblock {\em Preprint, {\em arXiv:2101.10272v3 [math.NT]}}.

\bibitem[MSS]{Mane:Sad:Sullivan}
R.~Ma\~{n}{\'e}, P.~Sad, and D.~Sullivan.
\newblock {On the dynamics of rational maps}.
\newblock {\em Ann. Sci. Ec. Norm. Sup.} {\bf 16}(1983), 193--217.

\bibitem[MaS]{Mavraki:Schmidt}
Niki~Myrto Mavraki and Harry Schmidt.
\newblock {On the dynamical Bogomolov conjecture for families of split rational
  maps}.
\newblock {\em Preprint, {\em arXiv:2201.10455v3 [math.NT]}}.

\bibitem[Mc1]{McMullen:families}
Curtis~T. McMullen.
\newblock {Families of rational maps and iterative root-finding algorithms}.
\newblock {\em Ann. of Math. (2)} {\bf 125}(1987), 467--493.

\bibitem[MeSc]{Medvedev:Scanlon}
Alice Medvedev and Thomas Scanlon.
\newblock {Invariant varieties for polynomial dynamical systems}.
\newblock {\em Ann. of Math. (2)} {\bf 179}(2014), 81--177.

\bibitem[MSW]{Miasnikov:Stout:Williams}
N.~Miasnikov, B.~Stout, and P.~Williams.
\newblock {Automorphism loci for the moduli space of rational maps}.
\newblock {\em Acta Arith.} {\bf 180}(2017), 267?296.

\bibitem[Mi1]{Milnor:quad}
John Milnor.
\newblock {Geometry and dynamics of quadratic rational maps}.
\newblock {\em Experiment. Math.} {\bf 2}(1993), 37--83.
\newblock With an appendix by the author and Lei Tan.

\bibitem[Mi2]{Milnor:Lattes}
John Milnor.
\newblock {On {L}att\`es maps}.
\newblock In {\em Dynamics on the {R}iemann sphere}, pages 9--43. Eur. Math.
  Soc., Z\"urich, 2006.

\bibitem[Ph]{Philippon1}
P.~Philippon.
\newblock {Sur des hauteurs alternatives. {I}}.
\newblock {\em Math. Ann.} {\bf 289}(1991), 255--283.

\bibitem[Pi]{Pink:conjecture}
Richard Pink.
\newblock {A combination of the conjectures of {M}ordell-{L}ang and
  {A}ndr\'e-{O}ort}.
\newblock In {\em Geometric methods in algebra and number theory}, volume 235
  of {\em Progr. Math.}, pages 251--282. Birkh\"auser Boston, Boston, MA, 2005.

\bibitem[Si1]{Silverman:Ratd}
Joseph~H. Silverman.
\newblock {The space of rational maps on $\bf {P}\sp 1$}.
\newblock {\em Duke Math. J.} {\bf 94}(1998), 41--77.

\bibitem[Si2]{Silverman:moduli}
Joseph~H. Silverman.
\newblock {\em Moduli spaces and arithmetic dynamics}, volume~30 of {\em CRM
  Monograph Series}.
\newblock American Mathematical Society, Providence, RI, 2012.

\bibitem[Ul]{Ullmo:Bogomolov}
Emmanuel Ullmo.
\newblock {Positivit\'e et discr\'etion des points alg\'ebriques des courbes}.
\newblock {\em Ann. of Math. (2)} {\bf 147}(1998), 167--179.

\bibitem[We]{West:moduli}
Lloyd West.
\newblock {The moduli space of cubic rational maps}.
\newblock {\em Preprint, {\em arXiv:1408.3247 [math.NT]}}.

\bibitem[Yu]{Yuan:uniform}
Xinyi Yuan.
\newblock {Arithmetic bigness and a uniform Bogomolov-type result}.
\newblock {\em Preprint, {\em arXiv:2108.05625v2 [math.NT]}}.

\bibitem[YZ]{Yuan:Zhang:quasiprojective}
Xinyi Yuan and Shouwu Zhang.
\newblock {Adelic line bundles over quasiprojective varieties}.
\newblock {\em Preprint, {\em arxiv:2105.13587v6 [math.NT]}}.

\bibitem[Za]{Zannier:book}
Umberto Zannier.
\newblock {\em Some problems of unlikely intersections in arithmetic and
  geometry}, volume 181 of {\em Annals of Mathematics Studies}.
\newblock Princeton University Press, Princeton, NJ, 2012.
\newblock With appendixes by David Masser.

\bibitem[Zh1]{Zhang:Bogomolov}
Shou-Wu Zhang.
\newblock {Equidistribution of small points on abelian varieties}.
\newblock {\em Ann. of Math. (2)} {\bf 147}(1998), 159--165.

\bibitem[Zh2]{Zhang:distributions}
Shou-Wu Zhang.
\newblock {Distributions in algebraic dynamics}.
\newblock In {\em Surveys in differential geometry. {V}ol. {X}}, volume~10 of
  {\em Surv. Differ. Geom.}, pages 381--430. Int. Press, Somerville, MA, 2006.

\end{thebibliography}
\def\cprime{$'$}

\end{document}